\documentclass[12pt]{amsart}
\usepackage[utf8]{inputenc}

\usepackage[english]{babel}
\usepackage{amsfonts,amsmath, amssymb, amsthm}
\usepackage{xcolor}
\usepackage{epsf,epsfig,a4wide}
\usepackage{url}

\newtheorem{Theorem}{Theorem}[section]
\newtheorem{Lemma}[Theorem]{Lemma}
\newtheorem{PrincipleLemma}[Theorem]{Principle Lemma}
\newtheorem{Corollary}[Theorem]{Corollary}

\theoremstyle{definition}
\newtheorem{Definition}[Theorem]{Definition}
\newtheorem{Remark}[Theorem]{Remark}
\newtheorem{Ex}[Theorem]{Example}

\begin{document}

\title[Structurally stable non-degenerate singularities]{Structurally stable non-degenerate singularities \\ of integrable systems}
\author[Kudryavtseva, Oshemkov]{E.~A.\ Kudryavtseva$^{*,**,1}$ and A.~A.\ Oshemkov$^{*,**,2}$}

\thanks{
Affiliations:\\
$^*$ Faculty of Mechanics and Mathematics, Moscow State University, Leninskie Gory 1, 119991Moscow, Russia, \\
$^{**}$ Moscow Center for Fundamental and Applied Mathematics, Leninskie Gory 1, 119991 Moscow, Russia \\
Email: $^1$\url{eakudr@mech.math.msu.su}, $^2$\url{oshemkov@mech.math.msu.su}}


\begin{abstract}
In this paper, we study singularities of the Lagrangian fibration given by a completely integrable system. We prove that a non-degenerate singular fibre satisfying the so-called connectedness condition is structurally stable under (small enough) real-analytic integrable perturbations of the system. In other words, the topology of the fibration in a neighbourhood of such a fibre is preserved after any such perturbation.
As an illustration, we show that a simple saddle-saddle singularity of the Kovalevskaya top is structurally stable under real-analytic integrable perturbations, but structurally unstable under $C^\infty$-smooth integrable perturbations.


\end{abstract}

\maketitle

\section{Introduction} \label {sec:1}

In this work, we study singularities of integrable systems. Recall that an integrable system is specified by a triple $(M^{2n},\omega,F)$, where $(M,\omega)$ is a symplectic $2n$-manifold and 
$$
F=(f_1,\dots,f_n):M\to\mathbb{R}^n
$$
is a {\em momentum map}, consisting of $n$ almost everywhere independent functions $f_i$ that pairwise Poisson commute: $\omega(X_{f_i},X_{f_j})=0$ for all $i,j=1,\dots,n$, where $X_{f_i}$ is defined by the rule $\omega(\cdot,X_{f_i})=\mathrm{d} f_i$. 

The momentum map $F$ naturally gives rise to a (singular) Lagrangian fibration on $M$ whose fibers are connected components of the common level sets $F^{-1}(a)$, $a\in\mathbb{R}^n$.
One can also write this fibration as the quotient map 
$$
\mathcal F: M \to B,
$$
where $B$ is the set of connected components of $F^{-1}(a)$, $a\in\mathbb{R}^n$, equipped with the quotient topology \cite{fom}. The space $B$ is usually referred to as the {\em bifurcation complex} (or the {\em unfolded momentum domain}) of the system.

\begin{Definition} \label{def:equiv}
Two integrable systems $(M_i,\omega_i,F_i)$, $i=1,2$, will be called {\em equivalent} 
(resp.\ {\em symplectically equivalent})
if there exists a homeomorphism 
(resp.\ symplectomorphism) $\Phi:M_1\to M_2$ and a homeomorphism $\phi:B_1\to B_2$ such that $\phi\circ\mathcal F_1=\mathcal F_2\circ\Phi$. The systems will be called {\em equivalent in a strong sense} 
or {\em left-right equivalent}
(resp.\ {\em symplectically equivalent in a strong sense})
if there exists a homeomorphism 
(resp.\ symplectomorphism) $\Phi:M_1\to M_2$ and a diffeomorphism $J:U_1\to U_2$ such that $J\circ F_1=F_2\circ\Phi$, for some neighbourhoods $U_i$ of $F_i(M_i)$ in $\mathbb{R}^n$, $i=1,2$. 
\end{Definition}

Consider the (local) Hamiltonian $\mathbb{R}^n$-action on $M$ generated by the momentum map $F$, i.e.\ by the (local) flows of the vector fields $X_{f_i}$. Orbits of this (local) action will be called simply {\em orbits}. 
Note that all fibers are invariant under the (local) flows of the vector fields $X_{f_i}$, thus each fiber is a union of orbits.
If the map $F$ is proper then the flows of the vector fields $X_{f_i}$ are complete, and we have a usual (well-defined) $\mathbb{R}^n$-action.

By a {\em singularity} of an integrable system, we will mean (following \cite{zung96, zung96a, zung00, zung03a, bol:osh06, han07, bgk}) the fibration germ at either a singular orbit (or its subset) or a singular fiber called {\em local} and {\em semilocal} {\em singularities}, resp. 
We note that, in literature, there is also the word ``semiglobal'' instead of our ``semilocal'', see e.g.\ \cite{vu2003}.
We recall that a point $m_0\in M$ is called a {\em singular} point of this fibration if $\operatorname{rank} \mathrm{d} F(m_0)<n$. An orbit is called {\em singular} if it contains a singular point (so, all its points are singular). A fiber is called {\em singular} if it contains at least one singular point; the minimal rank of singular points belonging to this fiber is called {\em rank} of the fiber.

Topology and geometry of integrable Hamiltonian systems with non-degenerate singularities have been studied from local \cite{vey, ler87, eli, ito91, zung96, zung96a, bau:zung97, zung:mir04, zung04},
semilocal \cite {fom, ler:uma2, zung96, zung96a, BF, ler00, kud:lep11, kud12} and global viewpoints \cite{fom, BF, bol:fom:ric, zung04}.
N.T.\ Zung developed a semilocal topological classification of  non-degenerate singularities \cite {zung96, zung96a}, and reduced a global topological (resp.\ symplectic) classification to rough topological (resp.\ symplectic) classification for ``generic'' integrable systems with singularities \cite {zung03a, zung04}.

\subsection{Structurally stable singularities}

Our central object will be {\em structurally stable} singularities. 
Informally speaking, a singularity is called structurally stable if the topology of the fibration is preserved after any (small enough) real-analytic integrable perturbations of the system.
Let us proceed with precise formulations.

In this paper, we assume that the manifold $M$, the symplectic structure $\omega$ and the momentum map $F$ are {\em real-analytic}. In the following definition,
$\|\ \|_0$ denotes the $C^0$-norm on the space of real-analytic pairs $(\omega^{\mathbb C},F^{\mathbb C})$ on $U_0^{\mathbb C}$.
Here $U_0^{\mathbb C}$ denotes a (small) open complexification of a neighbourhood $U_0$, while $\omega^{\mathbb C},F^{\mathbb C}$ are holomorphic extensions of $\omega,F$ to $U_0^{\mathbb C}$.

\begin{Definition} [{\cite[Def.~4.1]{kud:toric}}] \label {def:stab}
A compact subset $K$ of a singular orbit or fiber (and the singularity at $K$) of an integrable system $(M,\omega,F)$ will be called {\em structurally stable} (resp.\ {\em symplectically structurally stable}) if $K$ has a neighbourhood $U_0$ and its (small) open complexification $U_0^{\mathbb C}$ 
such that, for any smaller neighbourhood $U_1$ with a compact closure $\overline{U_1}\subset U_0$, there exists $\varepsilon>0$ 
satisfying the following condition: for any real-analytic integrable perturbation $(U_0,\tilde\omega,\tilde F)$ of $(U_0,\omega|_{U_0},F|_{U_0})$ such that $\|\tilde\omega^{\mathbb C}-\omega^{\mathbb C}\|_0+\|\tilde F^{\mathbb C}-F^{\mathbb C}\|_0<\varepsilon$, the integrable systems $(U,\omega|_U,F|_U)$ and $(\tilde U,\tilde\omega|_{\tilde U},\tilde F|_{\tilde U})$ are equivalent (resp.\ symplectically equivalent), cf.\ Definition \ref {def:equiv}, for some neighbourhoods $U,\tilde U\subseteq U_0$ containing $U_1$.\footnote{The notion of structural stability is known for $C^1$ vector fields (or flows) that are defined on a compact domain $U_0$ and satisfy a transversality condition on the boundary of $U_0$, in which case one has $U_1=U=\tilde U=U_0$.
However, for fibration germs, the domain $U_0$ is unfixed, and a transversality condition on $\partial U_0$ is often not fulfilled. We overcome these difficulties by using $U_1,U,\tilde U$.} In a similar way, structural stability is defined for an arbitrary compact subset $K$ of $M$.

If the integrable systems $(U,\omega|_U,F|_U)$ and $(\tilde U,\tilde\omega|_{\tilde U},\tilde F|_{\tilde U})$ are equivalent (resp.\ symplectically equivalent) in a strong sense, the singularity will be called {\em structurally stable} (resp.\ {\em symplectically structurally stable}) {\em in a strong sense}.
In a similar way, one defines {\em structural stability under integrable perturbations of some class}, e.g.\ the classes of {\em $C^\infty$ perturbations}, {\em $G$-symmetry-preserving perturbations} (where $G$ is a symmetry group of the singularity), {\em parametric perturbations} (analytically or smoothly depending on a small parameter) etc.
\end{Definition}

A Morse critical point of a smooth function on a surface can be viewed as a simplest singularity of integrable Hamiltonian systems with 1 d.f. 
It is well known that, due to the Morse lemma, Morse critical points are structurally stable,
moreover they are symplectically structurally stable \cite{cdv:vey79}. Non-degenerate singularities (cf.\ Sec.~\ref {subsec:nondeg}) are natural generalization of Morse critical points, and locally they are direct products of elliptic, hyperbolic, focus-focus and regular components (Theorems \ref {thm:1}, \ref {thm:2}). 

As we noted above, non-degenerate singularities have been extensively studied.
Nevertheless, some questions on structural stability of semilocal non-degenerate singularities remained open until now, and we give solutions to them in this paper in the real-analytic case (Theorems \ref {thm:stab:help}, \ref {thm:stab} and \ref {thm:stab:}, Corollary \ref {cor:stab:help}, Examples \ref {ex:complexity} and \ref {exa:koval}).

Below we mention some known results on structural stability of singularities:

1) Infinitesimal stability (i.e.\ stability under infinitesimal integrable deformations of the system \cite[Def.~8]{gia}) was studied for 2-degrees of freedom integrable systems, namely: non-degenerate rank-0 and rank-1 singular points and a rank-1 parabolic singular point are infinitesimally stable \cite[Def.~9, Theorems 2 and 3]{gia}.

2) Structural stability was proved \cite[Proposition 3.6]{izos:diss} for focus fibres of any dimension, satisfying connectedness condition (iii, iv) (or (v, vi)) of Theorem \ref {thm:stab:help} (such singularities were called {\em irreducible} in \cite{izos:diss}).

3) Structural stability under ``component-wise'' $C^\infty$ integrable perturbations was proved \cite{osh-tuz} for saddle-saddle fibers satisfying connectedness condition (v, vi) of Theorem \ref {thm:stab:help}.

4) Symplectic structural stability in a strong sense of non-degenerate compact orbits is known in real-analytic case \cite[Example 4.2 (A)]{kud:toric} (see also Theorem \ref {thm:stab:}).

5) In contrast to elliptic singularities (which are symplectically structurally stable due to the Eliasson Theorem \ref{thm:1}), simple semilocal singularities of hyperbolic and focus-focus types are symplectically structurally unstable. This follows from the presence of ``moduli'' in their symplectic classifications \cite {dmt94, vu2003}. Moreover, ``moduli'' also appear even in smooth classification for some classes of focus singularities of arbitrary dimension \cite{bol:izos}, which therefore are smoothly (and, hence, symplectically) structurally unstable.

6) For a parabolic singular point (cf.\ $I$ in Fig.~\ref {fig:koval} (b)), structural stability and $C^\infty$-smooth structural stability follow from \cite{ler87}. In the analytic case, symplectic structural stability of a parabolic point follows from \cite[Theorem 3]{var:giv82} (note that a parabolic point is infinitesimally non-degenerate, see \cite[Theorem 5.25]{zol} for a proof).

7) Parabolic orbits and cuspidal tori are structurally stable due to \cite{ler87}, moreover they are $C^\infty$-smoothly structurally stable due to \cite{kud:mar21}. However they are not symplectically structurally stable, because of the presence of ``moduli'' in their symplectic classifications (see \cite {bgk} for real-analytic case, \cite {kud:mar21a} for the smooth and real-analytic cases).

8) Structural stability under integrable perturbations preserving a Hamiltonian $(S^1)^{n-1}$-action (for $n$-degree of freedom integrable systems) was proved for many degenerate local singularities, e.g., parabolic orbits with resonances \cite{kal} (which are smoothly structurally stable when the resonance order is different from $4$ \cite{han07, kud:toric}), their parametric bifurcations \cite{han07}, periodic integrable Hamiltonian Hopf bifurcation \cite{vdm, han07} and its hyperbolic analogue \cite[Sec.~2]{ler00}, periodic integrable Hamiltonian Hopf bifurcations with resonances and their parametric bifurcations \cite{dui}, normally-elliptic parabolic orbits \cite{bro93}, normally-hyperbolic parabolic orbits etc. 
The above $F$-preserving Hamiltonian $(S^1)^{n-1}$-action is generated by $n-1$ functions, some of which are real-analytic functions multiplied with $\sqrt{-1}$ (in the real-analytic case) \cite[Example 3.12]{kud:toric}.
It is conjectured in \cite[Example 4.2 (B)]{kud:toric} that, using ``hidden'' torus actions, one can prove structural stability in a strong sense of the singular orbits mentioned above in real-analytic case.

9) Structural stability under real-analytic {\em parametric integrable perturbations} can be proved for the singularities mentioned in item 7 from above, via the convergence of the Birkhoff normal form \cite {zung05} and its analytic dependence on the perturbation parameters. 

In this paper, we prove (Theorem \ref {thm:stab}) that a non-degenerate semilocal singularity is structurally stable under real-analytic integrable perturbations, 
provided that it satisfies the connectedness condition (Definition \ref {def:conn}). We also give several criteria (Theorem \ref {thm:stab:help} and Corollary \ref {cor:stab:help}) for a semilocal singularity to satisfy the required assumptions (connectedness condition and non-degeneracy). As an illustration, we show that a saddle-saddle singularity of the Kovalevskaya top (and an arrangement of semilocal singularities containing this singularity) is structurally stable under real-analytic integrable perturbations, but structurally unstable under $C^\infty$-smooth integrable perturbations (Example \ref {exa:koval}).

\subsection {Non-degenerate singularities: local symplectic normal form} \label{subsec:nondeg}

Sufficient conditions for structural stability of a singularity are given in Theorems \ref {thm:stab} and \ref {thm:stab:}.
For their formulation, let us recall the notion of a non-degenerate singularity.

A singular point $m_0$ of rank $0$ is called {\em non-degenerate} (cf.\ e.g.\ \cite{eli, des90, ler94, zung96a, BF}) if the linearizations $A_j$ of the Hamiltonian vector felds $X_{f_j}$ at the singular point span a Cartan subalgebra of the Lie algebra of the Lie group $\operatorname{Sp}(T_{m_0}M,\omega|_{m_0})\simeq\operatorname{Sp}(2n,\mathbb{R})$, i.e.\ the operators $A_1,\dots,A_n$ span an $n$-dimensional commutative subalgebra and there exists a linear combination $A=\sum\limits_{j=1}^nc_jA_j$, $c_j\in\mathbb{R}$, having a simple spectrum:
$|\operatorname{Spec} A|=2n.$
A singular point $m_0$ of rank $r$ is called {\em non-degenerate} (cf.\ e.g.\ \cite{des90}) if the  rank-$0$ singular point of the corresponding reduced integrable Hamiltonian system with $n-r$ degrees of freedom (obtained by local symplectic reduction under the action of $f_1,\dots,f_r$ such that $\mathrm{d} f_1\wedge\dots\wedge \mathrm{d} f_r|_{m_0}\ne0$) is non-degenerate.
A singular orbit (respectively, fiber) is called {\em non-degenerate} if each singular point contained in this orbit (fiber) is non-degenerate.
Notice that, for an orbit, this condition holds automatically if at least one of its points is non-degenerate, but for a fiber it is not the case. 

The following assertion (known as Eliasson's theorem) is formulated for reader's interest, but it is not used any further in this article. Its proof is known for non-degenerate corank 1, elliptic, and focus-focus corank 2 singularities \cite{cdv:vey79, eli, vu:wac2013, zung:mir04} (more specifically, the rank 0 case was treated in \cite {cdv:vey79, eli, vu:wac2013}, resp., and the general case follows from rank 0 case due to \cite[Corollary 3.5]{zung:mir04}).
It is not clear whether there exists in the literature a complete proof of this assertion in the general case for singularities of all types.

\begin{Theorem} [Smooth local normal form] \label{thm:1}
For each non-degenerate singular point $m_0\in M$, the fibration is locally symplectically equivalent to the direct product of a regular fibration and several copies of elliptic, hyperbolic and focus-focus singularities, i.e., to a canonical system
\begin{equation} \label {eq:eli}
\begin{array} {ll}
h_s=\lambda_s & \mbox{for } 1\le s\le r,\\
h_{r+j}=\frac12(x_j^2 + y_j^2) & \mbox{for } 1\le j\le k_e,\\
h_{r+j}=x_{j}y_{j} & \mbox{for } k_e+1\le j\le k_e+k_h,\\
h_{r+j}=x_{j}y_{j} + x_{j+1}y_{j+1} & \mbox{and} \\
h_{r+j+1}=x_{j+1}y_{j} - y_{j+1}x_{j} & \mbox{for } j=k_e+k_h+2i-1,\ 1\le i\le k_f, 
\end{array} 
\end{equation}
\begin{equation} \label {eq:eli:}
\omega_{can}=
\sum\limits_{s=1}^r\mathrm{d}\lambda_s\wedge \mathrm{d}\varphi_s
+\sum\limits_{j=1}^{n-r}\mathrm{d} x_j\wedge \mathrm{d} y_j.
\end{equation}
\end{Theorem}

Thus, the canonical momentum map is defined by {\em regular} components $h_j$ ($1\le j\le r$), {\em elliptic} components $h_{r+j}$ ($1\le j\le k_e$), {\em hyperbolic} components $h_{r+k_e+j}$ ($1\le j\le k_h$) and {\em focus-focus} pairs of components $h_{r+k_e+k_h+2j-1},h_{r+k_e+k_h+2j}$ ($1\le j\le k_f$).
We say that the singular point $m_0$ has {\em Williamson type} $(k_e,k_h,k_f)$ \cite[Def.~2.3]{zung96a}. Notice that $r+k_e+k_h+2k_f=n$ and $r$ is the rank of $m_0$.

In real-analytic case, Theorem \ref {thm:1} admits the following strengthening. 

\begin{Theorem} [Real-analytic local normal form] \label{thm:2}
In real-analytic case, for each non-degenerate singular point $m_0\in M$, 

{\rm(a)} There exists a neighbourhood $U$ of $m_0$, in which the system is symplectically equivalent in a strong sense to \eqref{eq:eli}, \eqref{eq:eli:}, i.e.\ there exist a real-analytic symplectomorphism 
$$
\Phi=(\lambda_1,\varphi_1,\dots,\lambda_r,\varphi_r,x_1,y_1,\dots,x_{n-r},y_{n-r}):(U,\omega)\hookrightarrow(\mathbb{R}^{2n},\omega_{can})
$$
and a real-analytic diffeomorphism germ $J=(J_{1},\dots,J_{n}):(\mathbb{R}^n,F(m_0))\to(\mathbb{R}^n,0)$ such that $\Phi(m_0)=0$ and the map $J\circ F\circ\Phi^{-1}=(h_1,\dots,h_n)$ has a canonical form \eqref{eq:eli}.

{\rm(b)} If the flows of $X_{f_1},\dots,X_{f_n}$ are complete on the orbit $\mathcal{O}$ of $m_0$, then this orbit is diffeomorphic to a cylinder $\mathbb{R}^{r_o}\times(S^1)^{r_c}$ with $r_o+r_c=r$, and 
there exist a neighbourhood $U(\mathcal{O})$ of $\mathcal{O}$, a real-analytic symplectomorphism 
$$
\Phi:(U(\mathcal{O}),\omega)\hookrightarrow (V/\Gamma,\omega_{can})
$$
and a real-analytic diffeomorphism germ $J=(J_{1},\dots,J_{n}):(\mathbb{R}^n,F(m_0))\to(\mathbb{R}^n,0)$ such that the map $J\circ F\circ\Phi^{-1}=(h_1,\dots,h_n)$ has a canonical form \eqref{eq:eli} and $\Phi(\mathcal{O})=\{0\}^r\times\mathbb{R}^{r_o}\times(S^1)^{r_c}\times\{(0,0)\}^{n-r}$.
Here $V=D^r\times\mathbb{R}^{r_o}\times(S^1)^{r_c}\times(D^2)^{n-r}$ with the coordinates $(\lambda_s)_{s=1}^r$, $(\varphi_s)_{s=1}^r$ and $(x_j,y_j)_{j=1}^{n-r}$; 
$\Gamma$ is a finite group (called the {\em twisting group} at $\mathcal{O}$) that acts on $V$ freely and component-wise;
the action of $\Gamma$ on $(S^1)^{r_c}$ is by translations, its action on each hyperbolic disk $D^2$ is by multiplications by $\pm1$, its action on the remaining components is trivial (i.e.\ on $D^r$, on $\mathbb{R}^{r_o}$, on each elliptic disk $D^2$ and on each focus-focus polydisk $D^2\times D^2$), and its action on $(D^2)^{n-r}$ is effective.
\end{Theorem}

Thus, due to Theorem \ref {thm:2}, in analytic case, the momentum map itself is conjugated (left-right equivalent) to the canonical one.
In non-analytic case, the fibrations are the same, but the momentum maps are not necessarily conjugated, even in the case of a point. Examples of such situations (so-called splittable singularities) can be found in \cite[Fig.~1.9, 1.10, 9.63 and comments to and after them]{BF}, \cite [Sec.~5.3]{bol:osh06}.
This is why a description of structurally stable singularities in the smooth case is more difficult than in analytic case. In this paper, we consider analytic case only.

\begin{Definition} \label {def:Vey:momentum}
The composition $J\circ F$ from Theorem \ref {thm:2} (a) (resp.\ (b)) will be called {\em a Vey momentum map at the point} $m_0$ (resp.\ {\em at the orbit} $\mathcal{O}$). 
\end{Definition}

We note that the non-regular components of the Vey momentum map at a rank-$r$ point (after mutiplying some of them by $\sqrt{-1}$) generate a $(n-r)$-torus action near the point, which shows that these components are well-defined up to additive constants, but these constants can be uniquely chosen in order to make the values of the generating functions equal $0$ at fixed points of the action.

A proof of Theorem \ref {thm:2} is given in App.~\ref {sec:app}. 
Theorem \ref {thm:2} (a) was proved by J.~Vey \cite{vey}, and its equivariant generalization was proved by the first author \cite[Lemma 6.2]{kud:toric}. For a compact orbit $\mathcal{O}$, Theorem \ref {thm:2} (b) was proved in \cite[Theorem 2.1]{zung:mir04} for $C^\infty$ and real-analytic cases (see also \cite[Example 4.2 (A)]{kud:toric} for real-analytic case), and its equivariant generalization was proved \cite[Theorem 4.3]{zung:mir04} for $C^\infty$ and real-analytic cases.

\subsection{The connectedness condition} \label {subsec:conn}

Let $\mathcal{L}$ be a compact singular fiber (perhaps degenerate). 
In this subsection, we will assume that $\mathcal{L}$ is {\em almost non-degenerate} in the following sense: 
$\mathcal{L}$ consists of finitely many orbits, moreover if an orbit $\mathcal{O}_1$ is contained in the boundary of an orbit $\mathcal{O}\subset\mathcal{L}$ then rank of $\mathcal{O}_1$ is less than rank of $\mathcal{O}$.
It is clear that if $\mathcal{L}$ is non-degenerate then it is almost non-degenerate.

\begin{Definition} \label {def:conn}
We say that the singular fiber $\mathcal{L}$ (and the semilocal singularity at $\mathcal{L}$) of rank $r$ satisfies the {\em connectedness condition} if it contains a non-degenerate rank-$r$ orbit $\mathcal{O}_0\subset \mathcal{L}$
such that each of the $k_h+k_f$ subsets 
$$
\mathbb{K}_i:=\{m\in U(\mathcal{L})\mid \mathrm{d} (J_i\circ F)(m)=0\}, 
$$
$i\in\{r+k_e+a\}_{a=1}^{k_h}\cup\{r+k_e+k_h+2b\}_{b=1}^{k_f}$, 
is connected and contains any compact orbit $\mathcal{O}\subset \mathcal{L}$. Here $U(\mathcal{L})$ denotes a small neighbourhood of $\mathcal{L}$, $(k_e,k_h,k_f)$ is Williamson type of a point $m_0\in\mathcal{O}_0$, and $J\circ F$ is a Vey momentum map (Definition \ref {def:Vey:momentum}) at $m_0$ (we note that $J$ depends on the choice of the orbit $\mathcal{O}_0$). 
In practice, for verifying that $\mathbb{K}_i$ is connected, it is enough to check that, for each compact orbit $\mathcal{O}\subset\mathcal{L}$, $\mathbb{K}_i$ contains a path $\gamma_i$ joining $\mathcal{O}$ to $\mathcal{O}_0$.
\end{Definition}

Note that $\mathbb{K}_i=\{m\in U(\mathcal{L})\mid \mathrm{d} f_i(m)=0\}$, provided that the bifurcation diagram of $F|_{U(\mathcal{O}_0)}$ has a standard form (cf.\ Remark \ref {rem:a:b}).

Denote by $\mathbb{K}_{i}^{\mathcal{O}_0}$ the connected component of $\mathbb{K}_{i}$ containing $\mathcal{O}_0$.

\begin{Remark} \label {rem:conn}
We could require in Definition \ref{def:conn} that $\mathcal{O}$ is also contained in the subsets $\mathbb{K}_i$ for elliptic and/or focus-focus components $h_i$, $i\in\{r+a\}_{a=1}^{k_e}\cup\{r+k_e+k_h+2b-1\}_{b=1}^{k_f}$, of the Vey momentum map $J\circ F$ at $m_0\in\mathcal{O}_0$. But this requirement automatically holds, since $\mathbb{K}_i\supseteq \mathcal{L}$ for elliptic components $h_i$ (see \cite[Proposition 2.6]{zung96a}), and $\mathbb{K}_{i}\supseteq\mathbb{K}_{i+1}^{\mathcal{O}_0}$ for focus-focus pairs of components $h_i,h_{i+1}$.
\end{Remark}

\begin{Remark} \label {rem:tilde:K}
Observe that $\mathbb{K}_j^{\mathcal{O}_0}$ is a symplectic Bott critical submanifold of the function $J_j\circ F$ ($j>r$), since $\mathbb{K}_j^{\mathcal{O}_0}$ (resp., its complexification $(\mathbb{K}_j^{\mathcal{O}_0})^\mathbb{C}$) is the fixed point set of the Hamiltonian $S^1$-action generated by the function $J_j\circ F$ (resp., $iJ_j^\mathbb{C}\circ F^\mathbb{C}$) near $\mathbb{K}_j^{\mathcal{O}_0}$, due to Lemma \ref {lem:4} (a, b).
Obviously, all points of $\mathbb{K}_j$ are singular for the momentum map $F|_{U(\mathcal{L})}$. Thus, 
the singular point set of $F|_{U(\mathcal{L})}$ contains $\bigcup\limits_i\mathbb{K}_i\supseteq\bigcup\limits_i\mathbb{K}_i^{\mathcal{O}_0}$.
If $\mathcal{L}$ is non-degenerate, the connectedness condition simply means that the singular point set of $F|_{U(\mathcal{L})}$ is exhausted by $\bigcup\limits_i\mathbb{K}_i^{\mathcal{O}_0}$ and $\mathbb{K}_i^{\mathcal{O}_0}=\mathbb{K}_i$.
Then the submanifold $\mathbb{K}_i$ consists of points of rank $\le\frac12\dim\mathbb{K}_i$, and it is the closure of its open subset $\mathbb{K}_i\setminus\bigcup\limits_{i'\ne i}\mathbb{K}_{i'}$ consisting of non-degenerate rank-$\frac12\dim\mathbb{K}_i$ singular points.
\end{Remark}

\begin{Definition} \label {def:bif:diagram}
The singular values set of the momentum map $F$ is called the {\em bifurcation diagram} of the momentum map. 
\end{Definition}

\begin{Definition} [{\cite[Def.~6.3]{zung96a}, \cite[Def.~9.7]{BF}, \cite{bol:osh06}}] \label {def:non:split}
We say that the singular fiber $\mathcal{L}$ (and the semilocal singularity at $\mathcal{L}$) satisfies the {\em non-splitting condition}\footnote{Such singularities are also called {\em stable} \cite[Sec.~2.10]{fom91} or {\em topologically stable} \cite[Def.~4.5, 5.3, 6.3]{zung96a}.} if $\mathcal{L}$ is non-degenerate and the bifurcation diagram of the momentum map $F$ restricted to a small neighbourhood $U(\mathcal{L})$ of $\mathcal{L}$ coincides with the bifurcation diagram of the momentum map $F$ restricted to a small neighbourhood $U(\mathcal{O})$ of any compact orbit $\mathcal{O}\subseteq\mathcal{L}$.
\end{Definition}

\begin{Remark} \label {rem:a:b}
Suppose that $\mathcal{L}$ satisfies the non-splitting condition (Definition \ref {def:non:split}). Then all compact orbits in $\mathcal{L}$ have the same (minimal for $\mathcal{L}$) rank $r$ and the same Williamson type $(k_e,k_h,k_f)$, which will be called {\em Williamson type of $\mathcal{L}$}. Moreover any orbit in $\mathcal{L}$ is diffeomorphic to $\mathbb{R}^{a+b}\times(S^1)^{r+b}$ and has Williamson type $(k_e,k_h-a,k_f-b)$, for some $a,b\in\mathbb{Z}_+$ \cite[Propositions 2.6 and 3.5]{zung96a}. Therefore (after replacing $F$ by a Vey momentum map $J\circ F$ at a compact orbit $\mathcal{O}_0\subset\mathcal{L}$), we can assume that the bifurcation diagram of $F|_{U(\mathcal{L})}$ has a standard form.
\end{Remark}

Examples of non-degenerate semilocal singularities which do not satisfy the non-splitting condition (and, hence, the connectedness condition, cf.\ 
Corollary \ref {cor:stab:help}) can be found in \cite[Fig.~9.60--9.63 and Comments to them]{BF}, \cite [Lemma 3, Remark 4] {kud:osh}, \cite{koz:osh}.

\begin{Theorem} \label {thm:stab:help}
Suppose that $\mathcal{L}$ is a compact, almost non-degenerate singular fiber. If
\begin{itemize}
\item[{\rm(i)}] the fiber $\mathcal{L}$ satisfies the connectedness condition (Definition \ref {def:conn}),
\end{itemize}
then the non-degeneracy of $\mathcal{L}$ is equivalent to the following:
\begin{itemize}
\item[{\rm(ii)}] the $k_e+k_h+k_f$ submanifolds $\mathbb{K}_i$, $i\in\{r+a\}_{a=1}^{k_e+k_h}\cup\{r+k_e+k_h+2b\}_{b=1}^{k_f}$, from Definition \ref {def:conn} and Remark \ref {rem:tilde:K} are pairwise transversal at every compact orbit $\mathcal{O}\subseteq \mathcal{L}$ and have symplectic pairwise intersections at $\mathcal{O}$;
\end{itemize}
moreover {\rm(i, ii)} implies the following:
\begin{itemize}
\item[{\rm(iii)}] the fiber $\mathcal{L}$ satisfies the non-splitting condition (Definition \ref {def:non:split}).
\end{itemize}

If $\mathcal{L}$ satisfies the non-splitting condition {\rm(iii)}, then the connectedness condition {\rm(i)} is equivalent to the following:
\begin{itemize}
\item[{\rm(iv)}] $\mathbb{K}_i$ is connected for all $i\in\{r+k_e+a\}_{a=1}^{k_h}\cup\{r+k_e+k_h+2b\}_{b=1}^{k_f}$ (or, equivalently, for all $i=r+1,\dots,n$, cf.\ Remark \ref {rem:conn}).
\end{itemize}

If the singularity at 
$\mathcal{L}$ is of almost-direct-product type topologically {\rm (cf.\ \cite[Def.~7.2]{zung96a} or \cite {bol:osh06})}, i.e.\ 
\begin{itemize}
\item[{\rm(v)}] $\mathcal{L}$ is non-degenerate and its small neighbourhood is equivalent to the quotient $(V_1\times\dots\times V_{n-k_f})/\Gamma$ of the direct product of $r$ regular, $k_e$ elliptic, $k_h$ hyperbolic and $k_f$ focus-focus semilocal singularities by a free component-wise action of a finite group $\Gamma$, 
\end{itemize}
then the connectedness condition {\rm(i)} is equivalent to the following:
\begin{itemize}
\item[{\rm(vi)}] the action of $\Gamma$ on each component $V_i$ is transitive on the set of its singular points.
\end{itemize}
\end{Theorem}

\begin{Corollary} [Criteria for connectedness condition and non-degeneracy] \label {cor:stab:help}
For any compact, almost non-degenerate singular fiber $\mathcal{L}$, the following conditions are equivalent:
\begin{itemize}
\item $\mathcal{L}$ is non-degenerate and satisfies the connectedness condition, 
\item {\rm(i)} and {\rm(ii)};
\item {\rm(iii)} and {\rm(iv)};
\item {\rm(v)} and {\rm(vi)},
\end{itemize}
where {\rm(i)}--{\rm(vi)} are the conditions from Theorem  {\rm\ref {thm:stab:help}}.
\end{Corollary}

In Theorem \ref {thm:stab:help} and in what follows, 
the simultaneous fulfillment of the conditions (i) and (ii) is denoted by (i, ii), and similarly for (iii, iv) etc. We note that, for saddle-saddle singularities, (iii, iv) simply means that both components of the $l$-type \cite{bol91, mat96, BF} of the singularity are connected. 
Also note that (iii) in (iii, iv) is essential, see an example in \cite[Comment to Fig.~9.60]{BF}.
In fact, (v, vi) appears in \cite{osh-tuz} for saddle-saddle singularities. The equivalence of (iii, iv) and (v, vi) for rank-$0$ focus singularities is proved in \cite[Proposition 3.4]{izos:diss}.

\begin{Remark} \label {rem:proof:thm:stab:help}
In Theorem \ref {thm:stab:help}, the following implications are easy: (i, iii)$\Longrightarrow$(ii, iv), (iii, iv)$\Longrightarrow$(i) and (i, v)$\Longleftrightarrow$(v, vi). 
Thus, for proving Theorem \ref {thm:stab:help}, it is enough to prove the implications (i, ii)$\Longrightarrow$(iii) and (i, non-degeneracy of $\mathcal{L}$)$\Longrightarrow$(ii). 
\end{Remark}

Proofs of Theorem \ref {thm:stab:help} and Corollary \ref {cor:stab:help} are given in App.~\ref{sec:proof:}.

By {\em complexity} of a compact fiber $\mathcal{L}$, we will mean the number of compact orbits contained in this fiber (actually, by Remark \ref {rem:a:b}, compact orbits in $\mathcal{L}$ coincide with orbits of the minimal rank, provided that the non-splitting condition holds). 

\begin{Ex} \label {ex:complexity}
A topological classification of semilocal non-degenerate singularities satisfying the non-splitting condition and having complexity $\le c$ is known for the following Williamson types $(k_e,k_h,k_f)$ (cf.\ Remark \ref {rem:a:b}):
\begin{itemize}
\item rank-$0$ {\em saddle} type, $k_e=k_f=0$, 
with $(k_h,c)=(2,2)$ \cite{ler:uma2, bol91, mat96} (43 topological types of saddle-saddle singularities, where ${\mathcal B}\times {\mathcal B}$ and $({\mathcal B}\times {\mathcal C}_2)/\mathbb{Z}_2$ appear in the Kovalevskaya top, $({\mathcal B}\times {\mathcal D}_1)/\mathbb{Z}_2$ appears in the Goryachev-Chaplygin-Sretenskii case, $({\mathcal C}_2\times {\mathcal C}_2)/\mathbb{Z}_2$ occurs in the Clebsch case), 
with $(k_h,c)=(3,1)$ \cite{kal:thesis}, \cite[Theorem 9.13]{BF} (32 topological types of saddle-saddle-saddle singularities), and 
with any $(k_h,c)$ \cite{osh2010, osh2011}; 
\item rank-$r$ {\em center-focus} type, $k_h=0$, with any $(r,k_e,k_f,c)$ \cite{izo11} (e.g.\ ${\mathcal F}_1$ appears in the Lagrange case, ${\mathcal F}_2$ occurs in the Clebsch case); 
\item rank-$0$ {\em saddle-focus} type, $k_e=0$, 
with $(k_h,k_f,c)=(1,1,1)$ \cite{ler00} and with $(k_h,k_f,c)=(1,1,3)$ \cite {koz:osh20} (11 topological types of saddle-focus singularities of complexity 2, and 21 topological types of saddle-focus singularities of complexity 3). 
\end{itemize}
For example, singularities of the following topological types satisfy the connectedness condition: ${\mathcal A},$ ${\mathcal B},$ ${\mathcal A}^*,$ ${\mathcal F}_1,$ four saddle-saddle \cite{ler:uma2} and two saddle-focus \cite{ler00} singularities of complexity 1 
\begin{equation} \label {eq:exa:conn} 
{\mathcal B}\times {\mathcal B}, \ 
({\mathcal B}\times {\mathcal C}_2)/\mathbb{Z}_2, \ 
({\mathcal B}\times {\mathcal D}_1)/\mathbb{Z}_2, \ 
({\mathcal C}_2 \times {\mathcal C}_2) / (\mathbb{Z}_2 + \mathbb{Z}_2),\
{\mathcal B}\times{\mathcal F}_1,\ 
({\mathcal B}\times{\mathcal F}_2)/\mathbb{Z}_2,
\end{equation}
eleven saddle-saddle singularities of complexity 2 \cite{osh-tuz}
\begin{equation} \label {eq:exa:conn:} 
({\mathcal D}_1\times {\mathcal D}_1)/\mathbb{Z}_2,\quad
({\mathcal P}_4\times {\mathcal P}_4)/D_4,\quad
({\mathcal C}_2\times {\mathcal C}_2)/\mathbb{Z}_2,\quad
({\mathcal C}_1\times {\mathcal I}_1)/\mathbb{Z}_4,\quad
({\mathcal K}_3\times {\mathcal K}_3)/(\mathbb{Z}_4 + \mathbb{Z}_2),
\end{equation}
\begin{equation*}
({\mathcal C}_1\times {\mathcal J}_1)/\mathbb{Z}_4,\quad
({\mathcal C}_1\times {\mathcal K}_3)/\mathbb{Z}_4,\quad
({\mathcal C}_1\times {\mathcal P}_4)/\mathbb{Z}_4,\quad
({\mathcal D}_1\times {\mathcal C}_2)/\mathbb{Z}_2,\quad
({\mathcal C}_2\times {\mathcal P}_4)/(\mathbb{Z}_2 + \mathbb{Z}_2)
\end{equation*}
(where the latter case corresponds to two different types of singularities),
three saddle-focus \cite{koz:osh20} and one 4 d.f.\ focus (cf.\ \cite[Proposition 3.5]{izos:diss}, \cite[Sec.~8]{izo11}) singularities of complexity 2 
\begin{equation} \label {eq:exa:conn::}
({\mathcal D}_1\times {\mathcal F}_2)/\mathbb{Z}_2,\quad
({\mathcal C}_1\times {\mathcal F}_4)/\mathbb{Z}_4,\quad
({\mathcal C}_2\times {\mathcal F}_2)/\mathbb{Z}_2, \quad (\mathcal F_2\times\mathcal F_2)/\mathbb Z_2, 
\end{equation}
and their almost-direct products with each other and with a regular component ${\mathcal W}_{reg} = D^1\times S^1$ (we note that an almost-direct product of singularities satisfying the connectedness condition, obviously, also satisfies it).
Here ${\mathcal A}^*:=({\mathcal B}\times {\mathcal W}_{reg})/\mathbb{Z}_2$. 
By abusing notations, we will preserve the notation ${\mathcal V}$ for a singularity of the topological type ${\mathcal V}\times {\mathcal W}_{reg}$.
Some elementary semilocal singularities are shown in Fig.~\ref {fig:atoms}.
\end{Ex}

\begin{figure}[htbp]
\begin{center}
\includegraphics[width=0.04\textwidth]{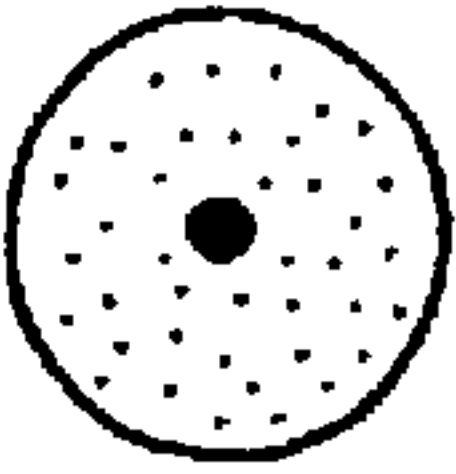}
\includegraphics[width=0.05\textwidth]{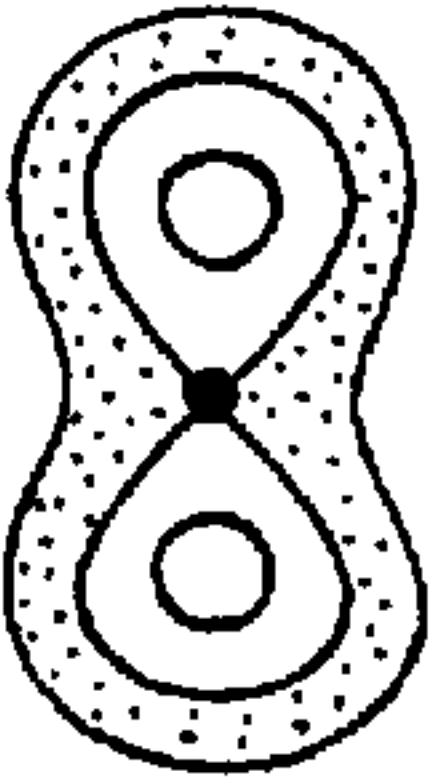}
\includegraphics[width=0.09\textwidth]{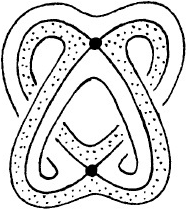}
\includegraphics[width=0.07\textwidth]{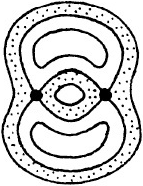}
\includegraphics[width=0.06\textwidth]{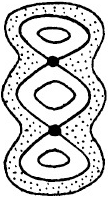}
\includegraphics[width=0.11\textwidth]{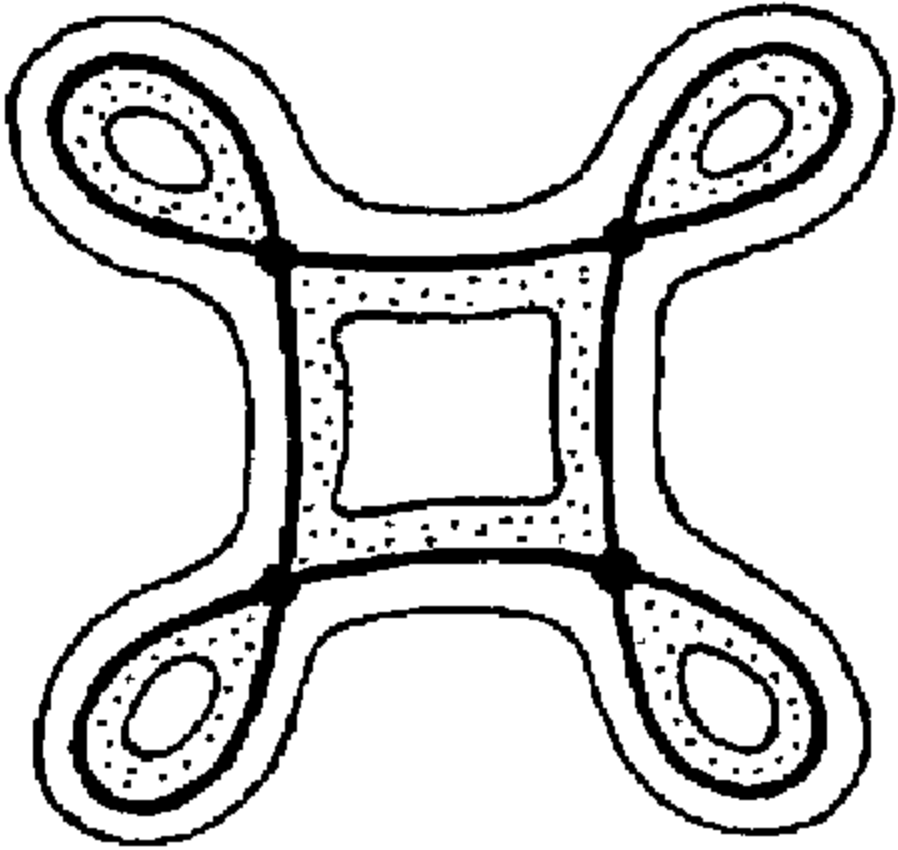}
\includegraphics[width=0.11\textwidth]{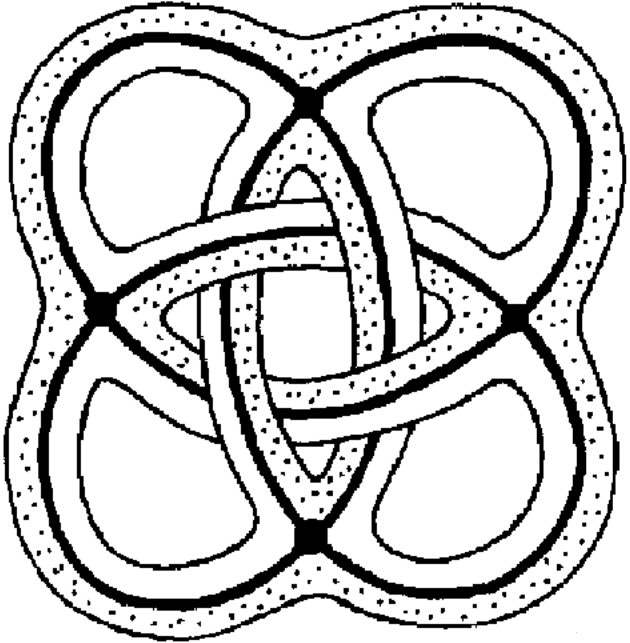}
\includegraphics[width=0.12\textwidth]{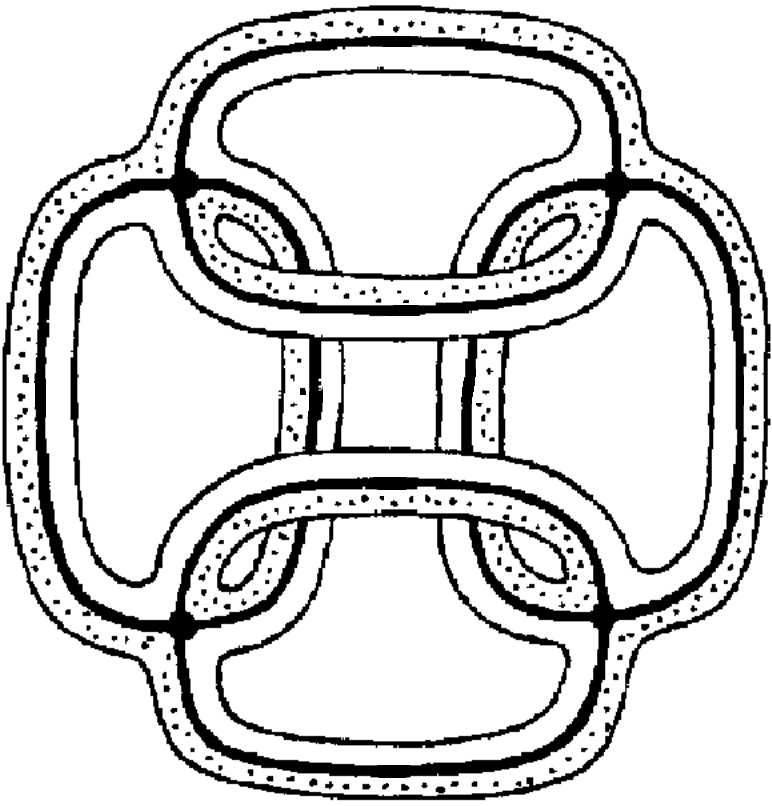}
\includegraphics[width=0.12\textwidth]{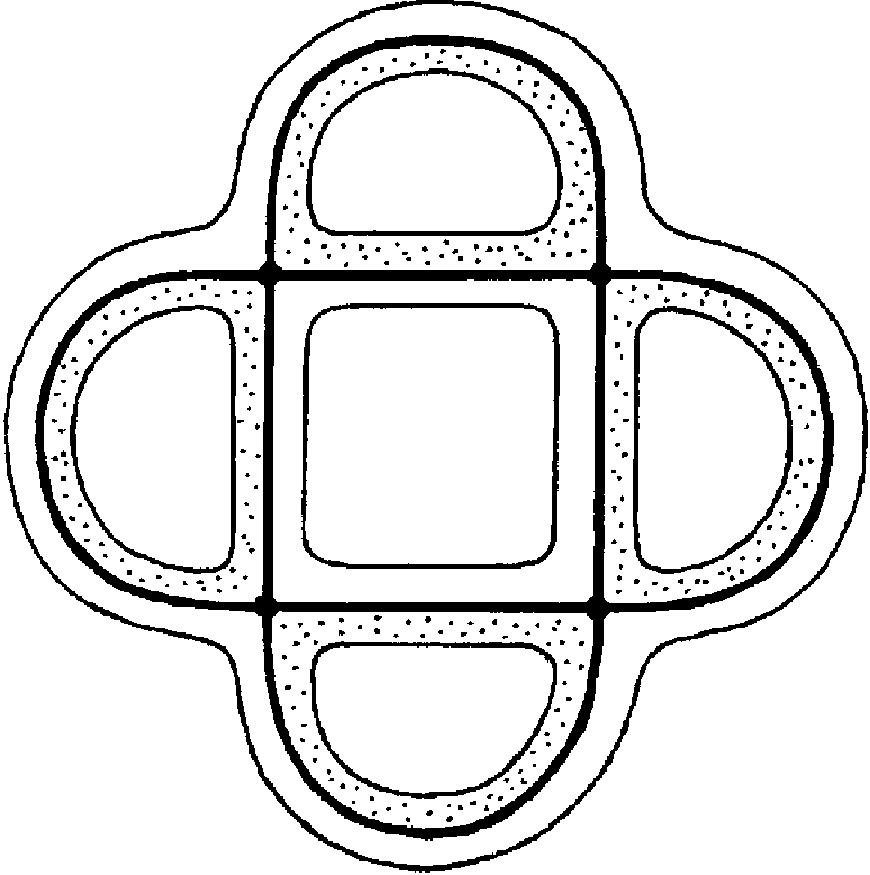}
\includegraphics[width=0.15\textwidth]{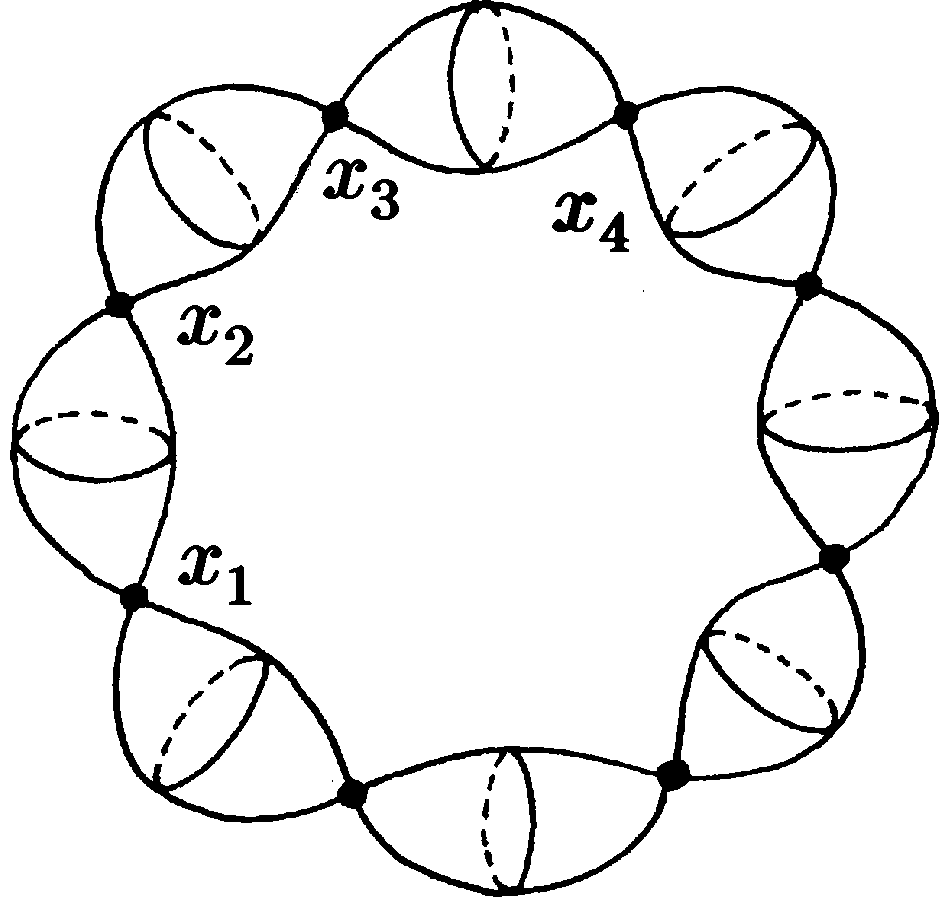}
\end{center}
${\mathcal A}\quad\ \ {\mathcal B}\qquad\ \ {\mathcal C}_1\qquad\ \ {\mathcal C}_2\qquad\ {\mathcal D}_1 \qquad \quad {\mathcal I}_1 \qquad\qquad {\mathcal J}_1 \qquad\ \ \qquad {\mathcal K}_3 \qquad \qquad \ {\mathcal P}_4 \qquad\qquad\quad\ {\mathcal F}_k \phantom{oooo}$
\caption {Elementary semilocal singularities: elliptic, some hyperbolic and focus-focus singularities.} \label {fig:atoms}
\end{figure}

\subsection{Main result}

\begin{Theorem} [Semilocal structural stability test] \label {thm:stab}
Suppose $\mathcal{L}$ is a compact non-degenerate singular fiber satisfying the connectedness condition (cf.\ Definition \ref {def:conn} and Corollary \ref {cor:stab:help}). Then the semilocal singularity at the fiber $\mathcal{L}$ is structurally stable in a strong sense (Definition \ref {def:stab}) under real-analytic integrable perturbations (but not necessarily under $C^\infty$ integrable perturbations). 
\end{Theorem}

Thus, by Theorem \ref {thm:stab:help} (v, vi) and Theorem \ref {thm:stab}, all singularities \eqref {eq:exa:conn}--\eqref{eq:exa:conn::}, as well as their almost-direct products with each other and with regular components, are structurally stable in a strong sense under integrable real-analytic perturbations. 

This theorem immediately implies the following.

\begin{Corollary} [Structural stability of simple singularities] \label {cor:stab}
Suppose $\mathcal{L}$ is a compact non-degenerate singular fiber containing a unique compact orbit $\mathcal{O}_0$.
Then the semilocal singularity at the fiber $\mathcal{L}$ is structurally stable in a strong sense under real-analytic integrable perturbations (but not necessarily under $C^\infty$ integrable perturbations, see Example \ref {exa:koval} for the saddle-saddle singularity $(\mathcal B\times\mathcal C_2)/\mathbb Z_2$). 
\end{Corollary}

A proof of Theorem \ref {thm:stab} is given in Sec.~\ref{sec:proof}. Theorem \ref {thm:stab} for rank-$0$ focus singularities satisfying connectedness condition (iii, iv) (or (v, vi)) of Theorem \ref {thm:stab:help} was proved in \cite[Proposition 3.6]{izos:diss}.

\section {Symplectic structural stability of local singularities} \label {sec:rank0}

\begin{Theorem} [Local symplectic structural stability] \label {thm:stab:}
Suppose $(M,\omega,F)$ is a real-analytic integrable system and we have one of the following situations:
\begin{itemize} 
\item[\rm(a)] $K$ is a non-degenerate singular point of the system,
\item[\rm(b)] $K$ is a compact subset of a non-degenerate singular orbit $\mathcal{O}$, and the flows of the vector fields $X_{f_1},\dots,X_{f_n}$ are complete on $\mathcal{O}$.
\end{itemize}
Then the singularity at $K$ (local singularity) is symplectically structurally stable in a strong sense (Definition \ref {def:stab}) under real-analytic integrable perturbations (but not necessarily under $C^\infty$ integrable perturbations).
\end{Theorem}

In Theorem \ref {thm:stab:}, we do not assume that the orbit $\mathcal{O}\supseteq K$ is compact. 
If $\mathcal{O}$ is compact, we can assume that $K=\mathcal{O}$. 
If $\mathcal{O}$ is non-compact, we can assume, e.g., that $K$ is a torus $K_0$ defined in the proof below. In Theorem \ref {thm:stab:}, we do not consider the case of a non-compact $K$ (e.g.\ when $\mathcal{O}$ is a non-compact singular orbit and $K=\mathcal{O}$), because even the notion of structural stability in Definition \ref {def:stab} is given only for compact subsets $K$.

\begin{proof}
Our proof follows \cite[Example 4.2(A)]{kud:toric} and is based on a strengthening (Lemma \ref {lem:period}) of the Vey Theorem \ref{thm:2} (see also \cite[Theorem 3.10]{kud:toric}).

{\em Step 1.} By Theorem \ref {thm:2}, each local non-degenerate singularity (a fibration germ at a point, or at an orbit) can be reduced to a normal form by a local symplectomorphism $\Phi$ and a local diffeomorphism $J$. Let us prove persistence of $\Phi$ and $J$ under (small) real-analytic integrable perturbations.

(a) Suppose $m_0\in M$ is a non-degenerate rank-$r$ point and $K=\{m_0\}$. Due to 
Lemma \ref {lem:period} (b), the local diffeomorphisms $\Phi$ and $J$ bringing the local momentum map at $m_0$ to the normal form  \eqref{eq:eli}, \eqref{eq:eli:} are persistent under real-analytic integrable perturbations.

(b) Suppose $\mathcal{O}\subseteq \mathcal{L}$ is a non-degenerate rank-$r$ orbit, and $K\subseteq\mathcal{O}$ its compact subset.
By Theorem \ref {thm:2} (b), there exist a neighbourhood $U(\mathcal{O})$ of $\mathcal{O}$ in $M$, a symplectomorphism $\Phi:(U(\mathcal{O}),\omega)\hookrightarrow (V/\Gamma,\omega_{can})$ and a diffeomorphism $J=(J_1,\dots,J_n):W\hookrightarrow \mathbb{R}^n$ such that $J\circ F\circ\Phi^{-1}$ has a canonical form \eqref{eq:eli}, where 
$V=D^r\times\mathbb{R}^{r_o}\times(S^1)^{r_c}\times(D^2)^{n-r}$ with the standard symplectic form \eqref{eq:eli:}, and $W\subset\mathbb{R}^n$ is a neighbourhood of $F(\mathcal{O})$ in $\mathbb{R}^n$. 
Since we can take a smaller neighbourhood if necessary, we can assume that we have the Hamiltonian $(S^1)^{n-r_o}$-action on a neighbourhood $U(\mathcal{O})^\mathbb{C}\supset U(\mathcal{O})$ of $\mathcal{O}$ in $M^\mathbb{C}$ generated by $n-r_o$ functions having the form $J_r^\mathbb{C}\circ F^\mathbb{C}$ and $iJ_j^\mathbb{C}\circ F^\mathbb{C}$.
Moreover, the orbit $\mathcal{O}$ is fixed under the $(S^1)^{n-r}$-subaction, and the $(S^1)^{r_c}$-subaction is locally-free on $\mathcal{O}$.
We want to show that $\Phi$ and $J$ are persistent on a neighbourhood of $K$ under integrable real-analytic perturbations.

Without loss of generality, we can and will assume that $K_0\subseteq K\subseteq K_1$, where $K_0$ is an $(S^1)^{r_c}$-orbit, and $K_1$ is the union of $r_c$-tori $\phi_{J_1\circ F}^{t_1}\circ\dots\circ\phi_{J_{r_o}\circ F}^{t_{r_o}}(K_0)$ forming a $r_o$-parameter family with parameters $t=(t_1,\dots,t_{r_o})\in B_{0,C}:=\{u\in\mathbb{R}^{r_o}\mid |u|<C\}$, 
for some fixed real value $C>0$ (here $\phi_f^t$ denotes the Hamiltonian flow generated by a function $f$).
Choose a neighbourhood $U_1$ of $K$ in $M^\mathbb{C}$ having a compact closure $\overline{U_1}\subset U(\mathcal{O})^\mathbb{C}$.

Now, we can follow the same arguments as in our proof of 
Lemma \ref {lem:period} (b) and Theorem \ref {thm:2} (b) (see App.~\ref {sec:app}). In this way, we see that the $(S^1)^{n-r_o}$-action and its normalization at the $r_c$-torus $K_0\subseteq\mathcal{O}$ in Theorem \ref {thm:2} (b) are persistent and rigid (resp.) under (small) integrable real-analytic perturbations. 
In detail, if the ``perturbation'' is $\varepsilon$-small, we can construct neighbourhoods $U,\tilde U\subseteq U(\mathcal{O})^\mathbb{C}$ of $K_0$ in $M^\mathbb{C}$ each of which contains $U_1$, a ``perturbed'' $(S^1)^{n-r_o}$-action on the neighbourhood $\tilde U$, a ``perturbed'' real-analytic Vey momentum map $\tilde J\circ\tilde F$ on $\tilde U\cap M$ and a ``perturbed'' real-analytic symplectomorphism $\tilde\Phi: (\tilde U\cap M,\tilde\omega)\hookrightarrow(V/\Gamma,\omega_{can})$ such that $\tilde\Phi(\tilde U\cap M)=\Phi(U\cap M)$ and 
$\tilde J\circ\tilde F\circ\tilde\Phi^{-1}$ has a canonical form \eqref{eq:eli}.
Moreover, the ``perturbed'' change $(\tilde\Phi,\tilde J)$ is $O(\varepsilon)$-close to the ``unperturbed'' change $(\Phi,J)$.

We can extend the symplectomorphism $\tilde\Phi$ to the ``perturbed'' neighbourhood $\tilde U(K):=\bigcup\limits_{t\in B_{0,C}}\phi_{\tilde J_1\circ\tilde F}^{t_1}\circ\dots\circ\phi_{\tilde J_{r_o}\circ\tilde F}^{t_{r_o}}(\tilde U\cap M)$ of $K_1$ using the ``perturbed'' Hamiltonian $\mathbb{R}^{r_o}$-action generated by $\tilde J_1\circ\tilde F,\dots,\tilde J_{r_o}\circ\tilde F$. Thus $\tilde J\circ\tilde F\circ\tilde\Phi^{-1}=(h_1,\dots,h_n)$ on the whole $\tilde U(K)$, as required. We also have $\tilde\Phi(\tilde U(K))=\Phi(U(K))$ for the similar ``unperturbed'' neighbourhood $U(K):=\bigcup\limits_{t\in B_{0,C}}\phi_{J_1\circ F}^{t_1}\circ\dots\circ\phi_{J_{r_o}\circ F}^{t_{r_o}}(U\cap M)$ of $K_1\supseteq K$.

{\em Step 2.} Thus $F|_{U(K)}$ and $\tilde F|_{\tilde U(K)}$ are conjugated via the symplectomorphism $\tilde\Phi^{-1}\circ\Phi:U(K)\to\tilde U(K)$ and the diffeomorphism $\tilde J^{-1}\circ J:W\to\tilde W$:
$$
\tilde F=(\tilde J^{-1}\circ J)\circ F\circ(\tilde\Phi^{-1}\circ\Phi)^{-1},
$$
which yields symplectic structural stability of the singularity at $K$ in a strong sense (Definition \ref {def:stab}).
\end{proof}

\begin{Definition} \label {def:Vey:momentum:}
The composition $\tilde J\circ\tilde F$ from the proof of Theorem \ref {thm:stab:} (a) (resp.\ (b)) will be called {\em a perturbed Vey momentum map near the point} $m_0$ (resp.\ {\em near the orbit} $\mathcal{O}$).
\end{Definition}

\section {Structural stability of semilocal singularities} \label {sec:proof}

In this section, we formulate Principle Lemma and derive Theorem \ref {thm:stab} from it. 

Let $\mathcal{L}\subset M$ be a compact non-degenerate fiber satisfying the connectedness condition. Let $\mathcal{O}_0\subseteq \mathcal{L}$ be an orbit of minimal rank in $\mathcal{L}$ satisfying the properties from Definition \ref {def:conn} of the connectedness condition. Due to Theorem \ref {thm:2} (b), we can define a Vey momentum map $J\circ F$ at $\mathcal{O}_0$ (Definition \ref {def:Vey:momentum}).
Observe that $\mathcal{O}_0$ is compact (otherwise its boundary contains a compact orbit of smaller rank, because $\mathcal{L}$ is almost non-degenerate, see Subsec.~\ref {subsec:conn}), thus it is symplectically structurally stable in a strong sense by Theorem \ref {thm:stab:}, so we can define a perturbed Vey momentum map $\tilde J\circ \tilde F$ near $\mathcal{O}_0$ (Definition \ref {def:Vey:momentum:}).

\begin{PrincipleLemma} \label {lem:3}
Under the above assumptions, every orbit $\mathcal{O}\subseteq \mathcal{L}$ has the following properties.

{\rm(a)} The Vey momentum map $J\circ F$ at the compact orbit $\mathcal{O}_0$ can serve as a Vey momentum map at the orbit $\mathcal{O}$ (Definition \ref {def:Vey:momentum}), with the same regular, elliptic, hyperbolic and focus-focus components apart from some of the hyperbolic and/or focus-focus components at $\mathcal{O}_0$ which are regular components at $\mathcal{O}$. 
If $\mathcal{O}$ is compact then it has the same rank and the same Williamson type as $\mathcal{O}_0$.

{\rm(b)} The perturbed Vey momentum map $\tilde J\circ\tilde F$ near $\mathcal{O}_0$ (Definition \ref {def:Vey:momentum:}) can serve as a perturbed Vey momentum map near $\mathcal{O}$.
\end{PrincipleLemma}

Thus, Principle Lemma \ref {lem:3} shows that, in order to construct Vey momentum maps for all orbits in $\mathcal{L}$, it is enough to construct it just for one compact orbit $\mathcal{O}_0$ in $\mathcal{L}$. Actually, Principle Lemma is very useful for symplectic classification of singularities under consideration, as we will show in the next work.

We remark that it is not hard to prove that, if the properties (a) and (b) hold for all compact orbits in $\mathcal{L}$, then they hold for all orbits in $\mathcal{L}$ (including non-compact ones), but proving the properties (a) and (b) for all compact orbits in $\mathcal{L}$ is more difficult and requires the assumptions on the fiber (namely, non-degeneracy and fulfillment of the connectedness condition). 

A proof of Principle Lemma is given in App.~\ref {sec:proof:}.
This lemma was proved in \cite[proof of Theorem~4.1]{izos:diss} for focus singularities satisfying connectedness condition (iii, iv) (or (v, vi)) of Theorem \ref {thm:stab}; it was used for proving that the equivalence and $C^\infty$-smooth equivalence actually coincide for such semilocal singularities \cite[Theorem~4.1]{izos:diss}.

\subsection {Proof of Theorem \ref {thm:stab}}

Let $U(\mathcal{O}_0)$ and $U(\mathcal{L})$ be small neighbourhoods of $\mathcal{O}_0$ and $\mathcal{L}$, resp.
For proving Theorem \ref {thm:stab}, it suffices to show that, if an integrable perturbation is small, then there exist a neighbourhood $\tilde U(\mathcal{L})$ of $\mathcal{L}$ close to $U(\mathcal{L})$ and a (perhaps, non-analytic) homeomorphism $\Psi:\tilde U(\mathcal{L})\to U(\mathcal{L})$ close to the identity such that $\tilde F=\tilde J^{-1}\circ J\circ F\circ\Psi$.

{\em Step 1.} Observe that the image and the ``unperturbed'' bifurcation diagram of a Vey momentum map $J\circ F$ at a compact orbit $\mathcal{O}_1\subseteq \mathcal{L}$ are standard and are completely determined by the Williamson type of $\mathcal{O}_1$. 
By Theorem \ref {thm:stab:}, every compact orbit is structurally stable in a strong sense under integrable real-analytic perturbations, thus its Williamson type is preserved under such perturbations.
Thus, Principle Lemma \ref{lem:3} implies that the ``unperturbed'' bifurcation diagram of $J\circ F|_{U(\mathcal{L})}$ is the same as the ``unperturbed'' bifurcation diagram of $J\circ F|_{U(\mathcal{O}_1)}$, and the same as the ``perturbed'' bifurcation diagram of $\tilde J\circ\tilde F|_{\tilde U(\mathcal{L})}$.
In particular, the perturbed semilocal singularity at $\tilde {\mathcal{L}}:=\tilde F^{-1}(\tilde J^{-1}(0))$ satisfies the non-splitting condition (Definition \ref {def:non:split}).

Without loss of generality, we can and will assume that $J$ and $\tilde J$ coincide with the identity.
Due to Principle Lemma \ref{lem:3}, for each singular orbit $\mathcal{O}\subset \mathcal{L}$, the perturbed fiber $\tilde {\mathcal{L}}$ contains a singular orbit $\tilde{\mathcal{O}}$ close to $\mathcal{O}$ and having the same rank, Williamson type and local bifurcation diagram as those of $\mathcal{O}$.

{\em Step 2.} Due to the Zung topological classification \cite[Theorem 7.3]{zung96a} of non-degenerate semilocal singularities satisfying the non-splitting condition, the semilocal singularity at $\mathcal{L}$ is equivalent to the almost-direct product of several semilocal singularities of the following types: regular, elliptic, hyperbolic and focus-focus ones. 

The proof of \cite[Theorem 7.3]{zung96a} uses the {\em $l$-type} of the singularity at $\mathcal{L}$, which is the (unordered) collection $(V_1,\dots,V_{k_e+k_h+k_f})$ of symplectic foliated $(2r+2)$- or $(2r+4)$-submanifolds $V_i=
\left(\bigcap\limits_{\substack{a=1 \\ a\ne i}}^{k_e+k_h}\mathbb{K}_{r+a}\right)\cap
\left(\bigcap\limits_{\substack{b=1 \\ k_e+k_h+b\ne i}}^{k_f}\mathbb{K}_{r+k_e+k_h+2b}\right) \subset U({\mathcal L})$ with singular fibres $K_i=V_i\cap {\mathcal L}$ (if $r=0$, then $(V_i,K_i)$ are so-called ``atoms'', may be non-connected).
Connected components of $K_i\setminus\cup(\mbox{compact orbits})$ called primitive orbits can be ``moved'' along each other  \cite[proof of Theorem 7.3]{zung96a}, provided that the given two primitive orbits ${\mathcal O}_1$ and ${\mathcal O}_2$ lie in different $V_i$ and in the closure of an orbit ${\mathcal O}\subset\mathcal{L}$ of dimension $\dim{\mathcal O}=\dim{\mathcal O}_1+\dim{\mathcal O}_2-r$, where $\partial{\mathcal O}={\mathcal O}_1+{\mathcal O}_2-{\mathcal O}_1'-{\mathcal O}_2'$ algebraically. Here, by moving ${\mathcal O}_1$ along ${\mathcal O}_2$, one get a primitive orbit ${\mathcal O}_1'$, which lies in the same $V_i$ as ${\mathcal O}_1$. By moving ${\mathcal O}_2$ along ${\mathcal O}_1$, one get a primitive orbit ${\mathcal O}_2'$, which lies in the same $V_j$ as ${\mathcal O}_2$. 
Let $C\subset U(\mathcal{L})$ be the union of all $V_i$ and of all orbits $\mathcal O\subset\mathcal{L}$ corresponding to the moves of primitive orbits (see above).
Let the tuple $(C, V_1,\dots,V_{k_e+k_h+k_f})$, equipped with inclusions $V_i\hookrightarrow C$ and some orientations, be called the {\em $Cl$-type} of the singularity at $\mathcal{L}$ \cite{BF}.
In the saddle-saddle case ($r=k_e=k_f=0$ and $k_h=2$), two singularities are equivalent if and only if their $Cl$-types are isomorphic \cite{mat96}. A key ingredient of \cite[proof of Theorem 7.3]{zung96a} is to show that the $Cl$-type of the singularity at $\mathcal{L}$ is isomorphic to the $Cl$-type of an almost-direct-product singularity.

One can deduce from Step 1 that the semilocal singularities at $\mathcal{L}$ and $\tilde {\mathcal{L}}$ have naturally isomorphic $Cl$-types. Applying the same 
arguments\footnote{Instead of arguments from \cite{zung96a}, we can apply Principle Lemma \ref {lem:3} for obtaining another proof of the fact that the singularities at $\mathcal L$ and $\tilde{\mathcal L}$ are equivalent in a strong sense.} 
as in \cite[proof of Theorem 7.3]{zung96a}, one obtains that these singularities are equivalent in a strong sense, as required. \qed

\section {Structurally stable non-degenerate singularities of the Kovalevskaya top} \label{sec:kov}

The base of the Liouville fibration is called the {\em bifurcation complex}; it was introduced by A.T.\ Fomenko (the cell-complex $K$ in \cite[Sec.~5]{fom},
the affine variety $A$ in \cite{efs:gia} called the {\em unfolded momentum domain}); it is a (branched) covering of the {\em momentum domain} $F(M)$ \cite{efs:gia}.
As pointed out by V.I.\ Arnold, it is interesting to investigate singularities of the bifurcation complex. 
S.P.\ Novikov stated the following important question: in which form does the bifurcation complex, as a topological invariant of the Liouville fibration, ``feels'' algebraicity of the momentum map $F$.

\begin{figure}[htbp]
\includegraphics[width=0.12\textwidth]{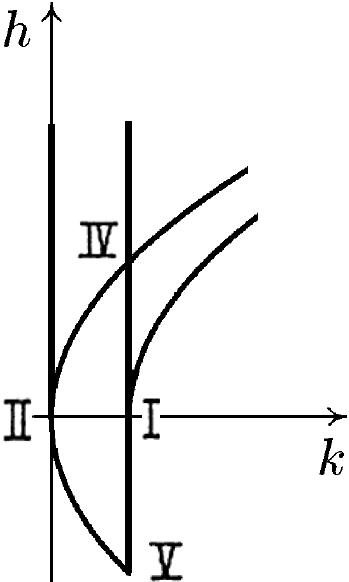} 
\includegraphics[width=0.25\textwidth]{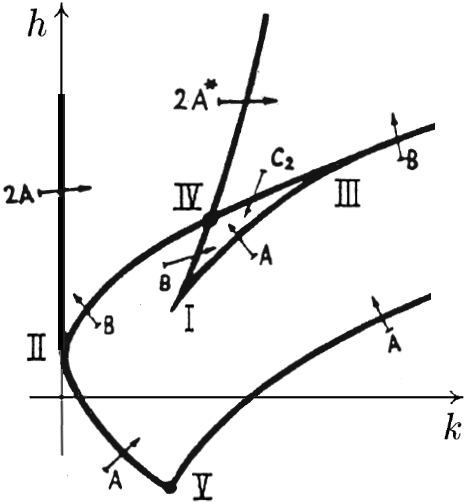} 
\includegraphics[width=0.2\textwidth]{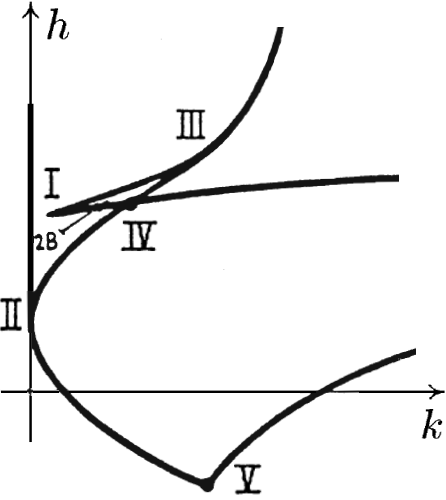}  
\includegraphics[width=0.19\textwidth]{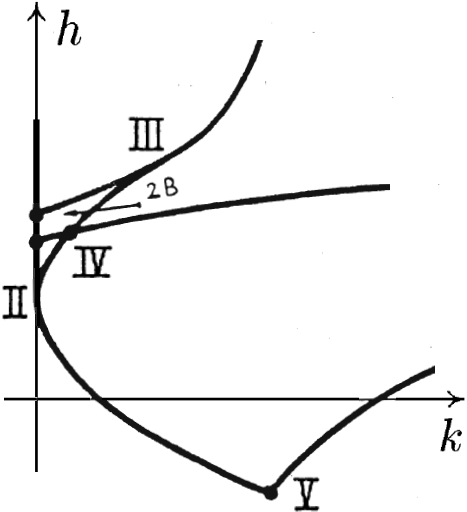} 
\includegraphics[width=0.2\textwidth]{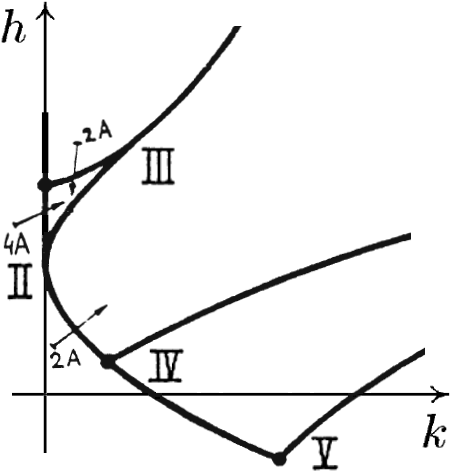}

{\small(a)} \qquad\qquad\qquad\quad {\small(b)} \qquad\qquad\qquad\qquad\ {\small(c)} \qquad\qquad\qquad\qquad {\small(d)} \qquad\qquad\qquad\quad {\small(e)} \phantom{oooo}
\caption {Bifurcation diagram and singularity types for the Kovalevskaya top with  (a) $g=0$, (b) $0<g^2<1$, (c) $1<g^2<8/(3\sqrt 3)$, (d) $8/(3\sqrt 3)<g^2<2$, (e) $g^2>2$.} \label {fig:koval}
\end{figure}

\begin{Ex} [Kovalevskaya's top] \label {exa:koval}
Consider the motion of a rigid body with a fixed point in a gravity field. The dynamical system is a Hamiltonian system on $e(3)^*$ which is $\mathbb{R}^6_{(R_1,R_2,R_3,S_1,S_2,S_3)}$ with the Poisson structure 
\begin{equation} \label {eq:poisson}
\{S_i,S_j\}=\varepsilon_{ijk}S_k, \quad \{R_i,R_j\}=0, \quad \{S_i,R_j\}=\{R_i,S_j\}=\varepsilon_{ijk}R_k,
\end{equation}
where $\{i,j,k\}=\{1,2,3\}$, and $\varepsilon_{ijk}$ is either the sign of the permutation $(i,j,k)$ if all the $i,j,k$ are different, or $0$ otherwise.
If the principle moments of inertia $(I_1,I_2,I_3)$ of the rigid body are proportional to $(2,2,1)$ and the center of masses lies in the plane perpendicular to the symmetry axis, the rigid body is called the {\em Kovalevskaya top}. The corresponding Hamiltonian system is integrable, with 
first integrals
$$
H=\frac12\left(S_1^2+S_2^2+2S_3^2\right)+R_1, \quad f_1=R_1^2+R_2^2+R_3^2, \quad f_2=S_1R_1+S_2R_2+S_3R_3,
$$
$$
K=\left(\frac{S_1^2}2-\frac{S_2^2}2-R_1\right)^2+(S_1S_2-R_2)^2
$$ 
(here $H$ is the Hamilton function, $f_i$ are Casimir functions of the Poisson bracket). 
Consider the restriction of this Kovalevskaya's system to a symplectic leaf $M^4_g=\{f_1=1,\ f_2=g\}$, where $g\in\mathbb{R}$ is the ``area constant'' regarded as a parameter of the system. 
The bifurcation diagrams (together with the bifurcation complexes) of the momentum maps $F_g=(H,K)|_{M^4_g}:M^4_g\to\mathbb{R}^2$ are shown in \cite {kha83} and in Fig.\ \ref {fig:koval}.

All local non-degenerate singularities are symplectically structurally stable in a strong sense under real-analytic integrable perturbations, due to Theorem \ref {thm:stab:}. Let us describe them.
Every open arc of the bifurcation diagram corresponds to one or two $1$-parameter families of rank-$1$ singular orbits of elliptic or hyperbolic types.
The vertices $IV$ and $V$ correspond to rank-$0$ singular points of hyperbolic-hyperbolic (if $1\ne g^2<2$), hyperbolic-elliptic (if $g^2>2$) and elliptic-elliptic types, respectively; these points form the fixed point set of the ${\mathbb Z}_2$-action on $M^4_g$ given by the canonical involution $(R_1,R_2,R_3,S_1,S_2,S_3)\mapsto(R_1,-R_2,-R_3,S_1,-S_2,-S_3)$ of the Kovalevskaya system.

All semilocal non-degenerate rank-$0$ singularities also happen to be structurally stable in a strong sense under real-analytic integrable perturbations (although not all under $C^\infty$ integrable perturbations), but some non-degenerate rank-$1$ singularities are not: 
\begin{itemize}
\item Due to \cite{kha83} or Corollary \ref {cor:stab}, all open arcs of the bifurcation diagram (Fig.~\ref {fig:koval}), except for ${\mathcal C}_2$, correspond to 1-parameter families of structurally stable in a strong sense (under real-analytic integrable perturbations) semilocal rank-$1$ singularities ${\mathcal A}, {\mathcal B}, {\mathcal A}^*$ (Fig.~\ref {fig:atoms}). However the semilocal rank-$1$ singularities corresponding to points of the open arc ${\mathcal C}_2$ (Fig.~\ref {fig:koval}, \ref {fig:C2:perturb}) are structurally unstable under $C^\infty$ integrable perturbations (this easily follows from the existence of a system-preserving free Hamiltonian $S^1$-action near such a semilocal singularity \cite[Theorem 4.1]{zung96a}). A related perturbation is shown in Fig.~\ref {fig:C2:perturb}.
\item Due to Corollary \ref {cor:stab}, the vertices $IV$ and $V$ correspond to structurally stable in a strong sense (under real-analytic integrable perturbations) semilocal rank-$0$ singularities. However the semilocal rank-$0$ singularity at $IV$ is structurally unstable under $C^\infty$ integrable perturbations if $g^2<1$ (since the singular fiber $IV$ is contained in the closure of the family of singular fibers ${\mathcal C}_2$, which are structurally unstable under $C^\infty$ integrable perturbations by above).
\end{itemize}
Note that the semilocal singularities corresponding to ${\mathcal C}_2$ (and probably to $IV$, topologically $({\mathcal B}\times {\mathcal C}_2)/\mathbb{Z}_2$, for $g^2<1$) are structurally stable under ${\mathbb Z}_2$-symmetry-preserving $C^\infty$ integrable perturbations, since ${\mathcal C}_2/{\mathbb Z}_2={\mathcal B}$ is structurally stable.

\begin{figure}[htbp]
\begin{center}
\includegraphics[width=0.2\textwidth]{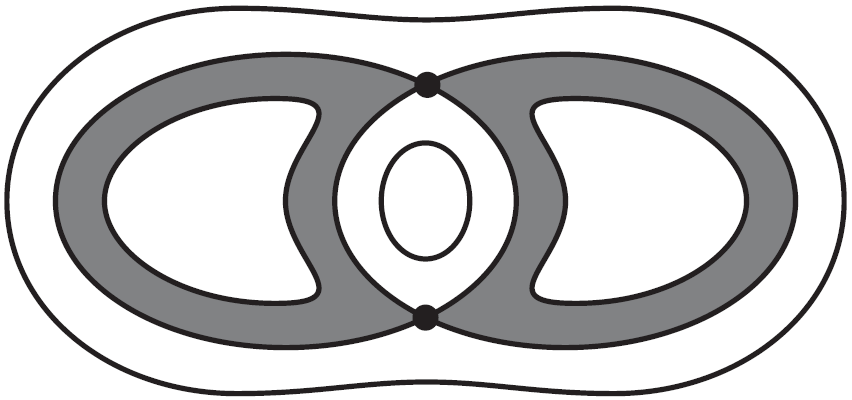}\qquad 
\includegraphics[width=0.2\textwidth]{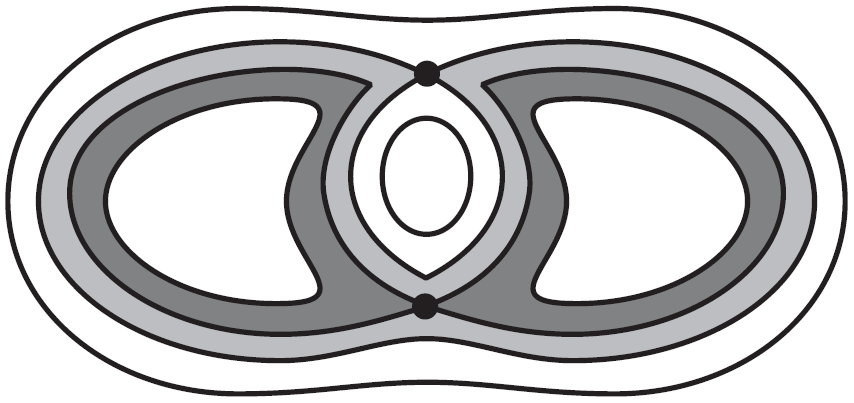} 

\phantom{oooooooooooooooo} $={\mathcal C}_2= \qquad\qquad\qquad\quad ={\mathcal B}-{\mathcal B}=$
\end{center}
\caption {Semilocal singularity ${\mathcal C}_2$ and its perturbed fibration (with notations due to \cite{BF}). White circles transform to gray ones through a singular fiber, and after that gray circles transform to black ones.} \label {fig:C2:perturb}
\end{figure}

The arrangement of semilocal singularities corresponding to any closed arc contained in the half-open arc ${\mathcal C}_2\cup\{IV\}$ and containing the point $IV$ for $g^2<1$ is structurally stable in a strong sense under real-analytic integrable perturbations (even under ${\mathbb Z}_2$-symmetry-breaking perturbations, although not under $C^\infty$ integrable perturbations, by above) and satisfies the assertion of Principle Lemma \ref {lem:3}. This can be proved using structural stability of $IV$ under real-analytic integrable perturbations (see above) and Principle Lemma \ref {lem:3} for $IV$, by noticing that the assertion of Principle Lemma can be extended from $IV$ to ${\mathcal C}_2\cup\{IV\}$ by analytic continuation.
\end{Ex}

\appendix
\section{Proof of Principle Lemma \ref {lem:3}, Theorem \ref {thm:stab:help} and Corollary \ref {cor:stab:help}} \label {sec:proof:}

Let $\mathcal{L}\subset M$ be a compact rank-$r$ fiber satisfying the connectedness condition (Definition \ref {def:conn}), and let $\mathcal{O}_0\subset \mathcal{L}$ be the corresponding non-degenerate rank-$r$ orbit.
Suppose Williamson type of $\mathcal{O}_0$ is $(k_e,k_h,k_f)$. 
By Theorem \ref {thm:2} (b), there exist a neighbourhood $U(\mathcal{O}_0)$ of $\mathcal{O}_0$, a local symplectomorphism $\Phi:U(\mathcal{O}_0)\hookrightarrow\mathbb{R}^{2n}$ at $\mathcal{O}_0$ and a real-analytic diffeomorphism $J=(J_1,\dots,J_n):W\hookrightarrow\mathbb{R}^n$ such that 
\begin{equation} \label {eq:canon}
J\circ F\circ\Phi^{-1} = (h_1,\dots,h_n)
\end{equation}
(cf.\ \eqref{eq:eli}, \eqref {eq:eli:}), so $J\circ F$ is a Vey momentum map at the orbit $\mathcal{O}_0$ (Definition \ref {def:Vey:momentum}). 

Since $\mathcal{O}_0$ is non-degenerate and compact, it is symplectically structurally stable in a strong sense by Theorem \ref {thm:stab:}. Thus, for a perturbed system, we have similar objects $\tilde U(\mathcal{O}_0)$, $\tilde\Phi$ and $\tilde J$ close to $U(\mathcal{O}_0)$, $\Phi$ and $J$ resp., where $\tilde\Phi:\tilde U(\mathcal{O}_0)\to\mathbb{R}^{2n}$ is a perturbed local symplectomorphism, and $\tilde J:\tilde W\to\mathbb{R}^n$ a perturbed diffeomorphism (as in the proof of Theorem \ref {thm:stab:}) near the compact rank-$r$ orbit $\mathcal{O}_0$.
In particular, we have a perturbed Vey momentum map $\tilde J\circ \tilde F$ near $\mathcal{O}_0$ (Definition \ref {def:Vey:momentum:}).

Denote $\mathbb{K}_i\cap \mathcal{L}$ by $K_i$ (for $r+1\le i\le n$).
Due to Remark \ref {rem:proof:thm:stab:help}, for proving Theorem \ref {thm:stab:help}, it is enough to prove the implications (i, ii)$\Longrightarrow$(iii) and (i, non-degeneracy of $\mathcal{L}$)$\Longrightarrow$(ii).
Principle Lemma \ref {lem:3} and Theorem \ref {thm:stab:help} readily follow from the following lemma.

\begin{Lemma} \label {lem:4}
Under the above assumptions, the following holds.

{\rm(a)} The functions 
$$
J_s\circ F, \quad 1\le s\le r+k_e, \qquad 
J_{r+k_e+k_h+2j}\circ F, \quad 1\le j\le k_f
$$
(corresponding to regular and elliptic components and elliptic parts of focus-focus components at $\mathcal{O}_0$, cf.\ \eqref{eq:canon}) generate a Hamiltonian $(S^1)^{r+k_e+k_f}$-action near the fiber $\mathcal{L}$ w.r.t.\ $\omega$.
The fiber $\mathcal{L}$ is fixed under the $(S^1)^{k_e}$-subaction of this action; each $\mathbb{K}_{r+i}$ is a Bott critical points set of the function $J_{r+i}\circ F$, $1\le i\le k_e$. The ``perturbed'' functions 
\begin{equation} \label {eq:periodic}
\tilde J_s\circ\tilde F, \quad 1\le s\le r+k_e, \qquad 
\tilde J_{r+k_e+k_h+2j}\circ\tilde F, \quad 1\le j\le k_f,
\end{equation}
generate a Hamiltonian $(S^1)^{r+k_e+k_f}$-action near the fiber $\mathcal{L}$ w.r.t.\ $\tilde\omega$.

{\rm(b)} If $K_j$ corresponds to a hyperbolic component $h_j$ in \eqref{eq:canon}, then $\mathbb{K}_j$ is a Bott critical points set of the function $J_j\circ F$, the function $iJ_j^\mathbb{C}\circ F^\mathbb{C}$ generates a Hamiltonian $S^1$-action near $K_j$ w.r.t.\ $\omega^\mathbb{C}$, and the function $i\tilde J_j^\mathbb{C}\circ\tilde F^\mathbb{C}$ generates a Hamiltonian $S^1$-action near $K_j$ w.r.t.\ $\tilde\omega^\mathbb{C}$. 

If $K_j$ corresponds to focus-focus components $h_{j-1},h_j$ in \eqref{eq:canon}, then $\mathbb{K}_j$ is a Bott critical point set of each function $J_{j-1}\circ F$ and $J_j\circ F$, the functions $iJ_{j-1}^\mathbb{C}\circ F^\mathbb{C}$ and $J_j^\mathbb{C}\circ F^\mathbb{C}$ generate a Hamiltonian $(S^1)^2$-action near $K_j$ w.r.t.\ $\omega^\mathbb{C}$, and the functions $i\tilde J_{j-1}^\mathbb{C}\circ\tilde F^\mathbb{C}$ and $\tilde J_j^\mathbb{C}\circ\tilde F^\mathbb{C}$ generate a Hamiltonian $(S^1)^2$-action near $K_j$ 
w.r.t.\ $\tilde\omega^\mathbb{C}$.

{\rm(c)} Suppose that $\mathcal{L}$ is almost non-degenerate (see Subsec.~\ref {subsec:conn}), and 
$\mathcal{L}$ is either non-degenerate or satisfies the condition {\rm(ii)} from Theorem \ref {thm:stab:help}. Suppose that an orbit $\mathcal{O}\subset \mathcal{L}$ has rank $r'$ and lies in $k_h'$ of the $k_h$ subsets $K_j$ corresponding to hyperbolic components $h_j$ and in $k_f'$ of the $k_f$ subsets $K_j$ corresponding to focus-focus components $h_{j-1},h_j$ in \eqref{eq:canon}.
Then $r'+k_e+k_h'+2k_f'=n$; $\mathcal{O}$ is non-degenerate of Williamson type $(k_e,k_h',k_f')$ and rank $r'=r+k_h-k_h'+2(k_f-k_f')\ge r$.
In particular, the $(S^1)^r$-action generated by $J_1\circ F,\dots,J_r\circ F$ (corresponding to regular components at $\mathcal{O}_0$) is locally free on $\mathcal{O}$ (and, hence, on $\mathcal{L}$). 
Furthermore, the map $J\circ F$ is a Vey momentum map at $\mathcal{O}$ (Definition \ref {def:Vey:momentum}) w.r.t.\ $\omega$ with the same regular, elliptic, hyperbolic and focus-focus components apart from $k_h-k_h'$ hyperbolic and $2(k_f-k_f')$ focus-focus components at $\mathcal{O}_0$ which are regular components at $\mathcal{O}$; 
the map $\tilde J\circ\tilde F$ is a perturbed Vey momentum map near $\mathcal{O}$ (Definition \ref {def:Vey:momentum:}) w.r.t.\ $\tilde\omega$.
\end{Lemma}

\begin{proof}
(a) Observe that the functions \eqref{eq:periodic} from the Vey presentation generate a Hamiltonian $(S^1)^{r+k_e+k_f}$-action on $\tilde U(\mathcal{O}_0)$ w.r.t.\ $\tilde\omega$, moreover the $(S^1)^r$-subaction is locally-free on $\tilde U(\mathcal{O}_0)$, and $\mathcal{O}_0$ is fixed under the unperturbed $(S^1)^{k_e+k_f}$-subaction.
Since $\mathcal{L}$ is compact, the time-$2\pi$ map of the flow of each vector field $X_{\tilde J_s\circ\tilde F}$, $1\le s\le r+k_e$ or $s=r+k_e+k_h+2j$, $1\le j\le k_f$, is well-defined on a neighbourhood of $\mathcal{L}$. Since this map is the identity on a $U(\mathcal{O}_0)$, and $\mathcal{L}$ is connected, it follows by uniqueness of analytic continuation that it is the identity on a neighbourhood of $\mathcal{L}$. Thus, the functions \eqref {eq:periodic} generate a Hamiltonian $(S^1)^{r+k_e+k_f}$-action on some neighbourhood of $\mathcal{L}$.

Due to \cite[Proposition 2.6]{zung96a}, the fiber $\mathcal{L}$ is fixed under the unperturbed $(S^1)^{k_e}$-subaction.

(b) Suppose that $h_j$ is a hyperbolic component in \eqref{eq:canon}. It follows from the Vey presentation \eqref {eq:eli}, \eqref{eq:eli:}, \eqref{eq:canon} that the function $iJ_j^\mathbb{C}\circ F^\mathbb{C}$ 
generates a Hamiltonian $S^1$-action on some neighbourhood $U(\mathcal{O}_0)^\mathbb{C}$ of $\mathcal{O}_0$ in $M^\mathbb{C}$ w.r.t.\ $\omega$. Similarly, the perturbed function $i\tilde J_j^\mathbb{C}\circ\tilde F^\mathbb{C}$ generates a Hamiltonian $S^1$-action on some neighbourhood $\tilde U(\mathcal{O}_0)^\mathbb{C}$ w.r.t.\ $\tilde\omega$. 

By definition of $K_j$, we have $\mathcal{O}_0\subseteq K_j\subseteq \mathcal{L}$ and $\mathrm{d}(J_j\circ F)=0$ at each point of $K_j$. Therefore, $\mathcal{O}_0$ is fixed under the $S^1$-action generated by $iJ_j^\mathbb{C}\circ F^\mathbb{C}$, and the time-$2\pi$ map of the flows of the vector fields $X_{iJ_j\circ F}$ and $X_{i\tilde J_j\circ\tilde F}$ are well-defined on some neighbourhood of $K_j$ in $M^\mathbb{C}$. Since this time-$2\pi$ maps are the identity on $U(\mathcal{O}_0)^\mathbb{C}$ and $K_j$ is connected, it follows by uniqueness of analytic continuation that these time-$2\pi$ maps are the identity on some neighbourhood of $K_j$. Thus, the functions $iJ_j\circ F$ and $i\tilde J_j\circ\tilde F$ generate Hamiltonian $S^1$-actions on some neighbourhoods $U(K_j)^\mathbb{C}$ and $\tilde U(K_j)^\mathbb{C}$ of $K_j$. We also showed that $(\mathbb{K}_j)^\mathbb{C}$ is the fixed points set on $U(K_j)$ of the unperturbed $S^1$-action, whence $\mathbb{K}_j$ is a symplectic submanifold and it is a Bott critical points set of $J_j\circ F$.

If $h_{j-1},h_j$ are focus-focus components in \eqref{eq:canon}, then similar arguments show that 
\begin{itemize}
\item the functions $iJ_{j-1}^\mathbb{C}\circ F^\mathbb{C},J_j^\mathbb{C}\circ F^\mathbb{C}$ generate a Hamiltonian $(S^1)^2$-action on some neighbourhood $U(K_j)^\mathbb{C}$ of $K_j$ in $M^\mathbb{C}$ w.r.t.\ $\omega^\mathbb{C}$, 
\item $(\mathbb{K}_j)^\mathbb{C}$ is the fixed points set of this $(S^1)^2$-action,
\item the perturbed functions $i\tilde J_{j-1}^\mathbb{C}\circ\tilde F^\mathbb{C}, \tilde J_j^\mathbb{C}\circ\tilde F^\mathbb{C}$ generate a Hamiltonian $(S^1)^2$-action on some neighbourhood $\tilde U(K_j)^\mathbb{C}$ of $K_j$ w.r.t.\ $\tilde\omega^\mathbb{C}$.
\end{itemize}

(c) By (a), the ``regular'' and the ``elliptic'' Vey functions $J_s\circ F$, $1\le s\le r+k_e$,
generate a Hamiltonian $(S^1)^{r+k_e}$-action on some neighbourhood of $\mathcal{L}$, and the $(S^1)^{k_e}$-subaction is fixed on $\mathcal{L}$.

By (b), we have $k_h'$ ``hyperbolic'' functions 
\begin{equation} \label{eq:*}
iJ_{\ell_j}^\mathbb{C}\circ F^\mathbb{C}, \qquad 1\le j\le k_h',
\end{equation}
and 
$2k_f'$ ``focus-focus'' pairs of functions 
\begin{equation} \label{eq:**}
iJ_{\ell_j-1}^\mathbb{C}\circ F^\mathbb{C}, \quad J_{\ell_j}^\mathbb{C}\circ F^\mathbb{C}, \qquad k_h'+1\le j\le k_h'+k_f',
\end{equation}
generating a Hamiltonian $(S^1)^{k_h'+2k_f'}$-action on some neighbourhood $U(\mathcal{O})^\mathbb{C}$ of $\mathcal{O}$, and this action is fixed on $\mathcal{O}$.

Let us first show that $r'+k_e+k_h'+2k_f'=n$ and $\mathcal{O}$ is non-degenerate of Williamson type $(k_e,k_h',k_f')$. Choose a point $m'\in\mathcal{O}$. Consider two cases.

{\em Case 1:} $\mathcal{O}$ is compact. Thus $k_h'=k_h$ and $k_f'=k_f$ by the connectedness condition. 
Thus, $m'$ is a fixed point of the Hamiltonian $(S^1)^{n-r}$-action on $U(\mathcal{O})^\mathbb{C}$ generated by the functions 
$J_{r+i}\circ F$, $1\le i\le k_e$,
$iJ_{r+k_e+j}^\mathbb{C}\circ F^\mathbb{C}$, $1\le j\le k_h$, and 
$iJ_{r+k_e+k_h+2j-1}^\mathbb{C}\circ F^\mathbb{C}, J_{r+k_e+k_h+2j}^\mathbb{C}\circ F^\mathbb{C}$, $1\le j\le k_f$.
Therefore $r'=\operatorname{rank}\mathrm{d} F(m')=\operatorname{rank}\mathrm{d}(J\circ F)(m')\le r=\operatorname{rank}\mathrm{d} F(m_0)$.
But $m_0$ has minimal rank on $\mathcal{L}$ by connectedness condition. Therefore $r'=r$, thus the functions $J_s\circ F$, $1\le s\le r$, generate a locally-free $(S^1)^r$-action on $\mathcal{O}$.

Thus $m'$ is a rank-$r$ point of the Hamiltonian $(S^1)^n$-action generated by the above functions (having the form $J_s\circ F$, $iJ_j\circ F$). 
By \cite[Theorem 3.10]{kud:toric}, there exists a real-analytic symplectomorphism $\Phi':(U(\mathcal{O}),\omega)\to(V/\Gamma',\omega_{can})$ such that 
$$
J_s\circ F\circ{\Phi'}^{-1}=h_s, \quad 1\le s\le r,
$$
the regular components $h_1,\dots,h_r$ on $V$, while 
$$
J_i\circ F\circ{\Phi'}^{-1} = \sum\limits_{j=r+1}^n k_{ij} h_j' , \qquad r+1\le i\le n,
$$
for some integers $k_{ij}$ and $n-r$ quadratic functions $h_i'$ of elliptic, hyperbolic and focus-focus types (the number of components $h_i'$ of each type is not necessary $k_e,k_h,k_f$ as in \eqref{eq:canon}).
As we showed above, each $\mathbb{K}_i^\mathbb{C}$ ($i>r$) is a fixed point set of the corresponding $S^1$-subaction (resp.\ $(S^1)^2$-subaction) of the $(S^1)^{n-r}$-action on $U(\mathcal{O})^\mathbb{C}$, and the type of this subaction is given by $h_i$ (resp.\ $h_{i-1},h_i$) in \eqref{eq:canon}. 

Since, by assumption, $\mathcal{L}$ either is non-degenerate or satisfies the condition (ii) from Theorem \ref {thm:stab:help}, we conclude that
$$
J_i\circ F\circ{\Phi'}^{-1}=h_i, \quad r+1\le i\le n,
$$
after changing $\Phi'$ if necessarily.
In particular, $\mathcal{O}$ is non-degenerate and has Williamson type $(k_e,k_h,k_f)$ and rank $r'=r$, moreover $J\circ F$ is a Vey momentum map at $\mathcal{O}$.

{\em Case 2:} $\mathcal{O}$ is noncompact. Thus its closure $\overline{\mathcal{O}}$ contains a compact orbit $\mathcal{O}_1\subset \mathcal{L}$ (because $\mathcal{L}$ is almost non-degenerate).
By Case 1, $\mathcal{O}_1$ is non-degenerate of rank $r$ and Williamson type $(k_e,k_h,k_f)$, moreover $\mathcal{O}_1$ lies in each $K_i$, $r+1\le i\le n$, and there exists a real-analytic symplectomorphism $\Phi_1:(U(\mathcal{O}_1),\omega)\to(V/\Gamma_1,\omega_{can})$ such that 
$J_i\circ F\circ\Phi_1^{-1}=h_i$, $1\le i\le n$.

Since $\mathcal{O}_1$ is non-degenerate and $\mathcal{O}_1\subset\overline{\mathcal{O}}$, we conclude that $\mathcal{O}$ is non-degenerate too, moreover (by Remark \ref {rem:a:b}) it has Williamson type $(k_e,k_h-a,k_f-b)$ and is diffeomorphic to $\mathbb{R}^{a+b}\times(S^1)^{r+b}$, for some $a,b\in\mathbb{Z}_+$. Since $\mathcal{O}_1$ lies in each $K_i$, $r+1\le i\le n$, it follows that $k_h'=k_h-a$ and $k_f'=k_f-b$. Thus $r'=r+a+2b$ and 
$$
r'+k_e+k_h'+2k_f'=r+k_e+k_h+2k_f=n,
$$
as required. We also obtain that the functions $J_i\circ F$, $1\le i\le r$, generate a locally-free $(S^1)^r$-action on $\mathcal{O}$.

It remains to show that $J\circ F$ is a Vey momentum map at $\mathcal{O}$ (Definition \ref {def:Vey:momentum}) w.r.t.\ $\omega$, and $\tilde J\circ\tilde F$ is a perturbed Vey momentum map near $\mathcal{O}$ (Definition \ref {def:Vey:momentum:}) w.r.t.\ $\tilde\omega$. 
On one hand, by (a) and (b), $\mathcal{O}$ is fixed under the Hamiltonian $(S^1)^{n-r'}$-action on $U(\mathcal{O})^\mathbb{C}$. 
On the other hand, as we showed above, $\mathcal{O}$ is contained in $k_e+k_h'+k_f'$ subsets $\mathbb{K}_i^\mathbb{C}$, $i\in\{r+1,\dots,r+k_e\}\cup\{\ell_j\}_{j=1}^{k_h'}$ (resp.\ $i\in\{\ell_j\}_{j=k_h'+1}^{k_h'+k_f'}$), see \eqref{eq:*}, \eqref{eq:**}, each of which is a fixed point set of the corresponding $S^1$-subaction (resp.\ $(S^1)^2$-subaction) of the $(S^1)^{n-r'}$-action on $U(\mathcal{O})^\mathbb{C}$, and the type of this subaction is given by $h_i$ (resp.\ $h_{i-1},h_i$) in \eqref{eq:canon}. 

But, by the assumption, $\mathcal{L}$ is non-degenerate or satisfies the condition (ii) from Theorem \ref {thm:stab:help}, therefore the $k_e+k_h'+k_f'$ symplectic submanifolds $\mathbb{K}_i\cap U(\mathcal{O})$ are pairwise transversal and have symplectic pairwise intersections at $m'\in\mathcal{O}$. It follows from 
Lemma \ref {lem:period} (a) that there exists a real-analytic symplectomorphism $\Phi':(U(m'),\omega)\to(\mathbb{R}^{2n},\omega_{can})$ such that the $n-r'=k_e+k_h'+2k_f'$ functions 
$$
\begin{array}{l}
J_{r+i}\circ F\circ{\Phi'}^{-1}, \quad 1\le i\le k_e,\qquad\qquad
J_{\ell_j}\circ F\circ{\Phi'}^{-1}, \quad 1\le j\le k_h', \\
J_{\ell_j-1}\circ F\circ{\Phi'}^{-1},\ J_{\ell_j}\circ F\circ{\Phi'}^{-1}, \qquad k_h'+1\le j\le k_h'+k_f',
\end{array}
$$
coincide with the quadratic functions $h_i$ of elliptic, hyperbolic and focus-focus types, resp., while the remaining $r'$ functions $J_{i}\circ F\circ{\Phi'}^{-1}$ (whose differentials are automatically linearly independent at $m'$, since $r'=\operatorname{rank}\mathrm{d}(J\circ F)(m')$) are linear functions $\lambda_1,\dots,\lambda_{r'}$.

Thus $\mathcal{O}$ is non-degenerate of Williamson type $(k_e,k_h',k_f')$ and rank $r'$, and the map $J\circ F$ is a Vey momentum map at $m'$.
In fact, we have even more: it is a Vey momentum map at $\mathcal{O}$ (Definition \ref {def:Vey:momentum}), since, due to (a), the remaining $r'=r+k_h-k_h'+2k_f-2k_f'$
functions (namely, the functions $J_s\circ F$, $1\le s\le r$, and the remaining $k_h-k_h'$ hyperbolic functions $J_j\circ F$ and $k_f-k_f'$ focus-focus pairs of functions $J_{j-1}\circ F, J_j\circ F$) generate an $\mathbb{R}^{k_h-k_h'+k_f-k_f'}\times(S^1)^{r+k_f-k_f'}$-action near $\mathcal{L}$, which is locally-free near $\mathcal{O}$, and we can use this action for extending the local symplectomorphism $\Phi'$ to a neighbourhood of the cylinder $\mathcal{O}\approx \mathbb{R}^{k_h-k_h'+k_f-k_f'}\times(S^1)^{r+k_f-k_f'}$, as in the proof of Theorem \ref {thm:2} (b).

Due to (a) and (b), the $n-r'=k_e+k_h'+2k_f'$ perturbed functions 
$\tilde J_{r+i}\circ\tilde F$, $1\le i\le k_e$,
$i\tilde J_{\ell_j}\circ\tilde F$, $1\le j\le k_h'$, and 
$i\tilde J_{\ell_j-1}\circ\tilde F,\tilde J_{\ell_j}\circ\tilde F$, $k_h'+1\le j\le k_h'+k_f'$, 
generate a ``perturbed'' Hamiltonian $(S^1)^{n-r'}$-action near $\mathcal{O}$. 
Since the map $\tilde J\circ\tilde F$ is close to $J\circ F$, which is a Vey momentum map at $\mathcal{O}$ by above, it follows from 
Lemma \ref {lem:period} (b) that $\tilde J\circ\tilde F$ is a perturbed Vey momentum map near $m'$ (Definition \ref {def:Vey:momentum:}). In fact, we have even more: it is a perturbed Vey momentum map near $\mathcal{O}$, since we can extend the corresponding local symplectomorphism $\tilde\Phi'$ to a neighbourhood of the cylinder $\mathcal{O}\approx \mathbb{R}^{k_h-k_h'+k_f-k_f'}\times(S^1)^{r+k_f-k_f'}$ using the perturbed locally-free $\mathbb{R}^{k_h-k_h'+k_f-k_f'}\times(S^1)^{r+k_f-k_f'}$-action generated by $\tilde J_s\circ\tilde F$, $1\le s\le r$, and the remaining $k_h-k_h'$ hyperbolic functions $\tilde J_j\circ\tilde F$ and $k_f-k_f'$ focus-focus pairs of functions $\tilde J_{j-1}\circ\tilde F, \tilde J_j\circ\tilde F$.
\end{proof}

\begin{proof}[Proof of Corollary \ref {cor:stab:help}]
We have to prove the equivalence of four conditions. It follows from Theorem \ref {thm:stab:help} that all of these conditions except for the last one are pairwise equivalent, and the last one implies the previous ones. Moreover the last one follows from the previous one, provided that (iii) implies (v). It is left to note that the latter implication is the Zung topological classification \cite[Theorem 7.3]{zung96a}.
\end{proof}

\section{Local normal form and its rigidity} \label {sec:app}

Here we give a proof of Theorem \ref{thm:2} using the following lemma, which we also use (in Sec.~\ref {sec:rank0} and App.~\ref {sec:proof:}) in the proofs of Theorems \ref {thm:stab:}, \ref {thm:stab:help} and Principle Lemma \ref {lem:3}.

\begin{Lemma} \label {lem:period}
Suppose $m_0\in M$ is a singular rank-$r$ point of a real-analytic integrable system $(M,\omega,F)$.
Suppose the first differentials of the functions $f_{r+1},\dots,f_n$ at $m_0$ vanish and, in some canonical chart $\Phi_0:(U_0,\omega)\hookrightarrow(\mathbb{R}^{2n},\omega_{can})$ with $\Phi_0(m_0)=0$, the second differentials of $f_{r+1}\circ\Phi_0^{-1},\dots,f_n\circ\Phi_0^{-1}$ 
at $0$ coincide with the second differentials of $h_{r+1},\dots,h_n$ in \eqref{eq:eli}, and $\mathrm{d}(f_s\circ\Phi_0^{-1})(0)=\mathrm{d}\lambda_s$ for $1\le s\le r$. Then

{\rm(a)} There exist a neighbourhood $U$ of $m_0$ in $M^\mathbb{C}$, a neighbourhood $W\supseteq F^\mathbb{C}(U)$ of $F(m_0)$ in ${\mathbb C}^n$, and a unique Hamiltonian $(S^1)^{n-r}$-action on $U$ generated by the functions 
\begin{equation} \label {eq:X:}
\begin{array}{c}
J_{r+1}\circ F,\dots,J_{r+k_e}\circ F, \quad
iJ_{r+k_e+1}\circ F,\dots,iJ_{r+k_e+k_h}\circ F, \\
iJ_{r+k_e+k_h+1}\circ F,J_{r+k_e+k_h+2}\circ F,\quad \dots,\quad 
iJ_{n-1}\circ F,J_n\circ F,
\end{array}
\end{equation}
for some real-analytic functions $J_j:W\to\mathbb{C}$ with $J_j(z)=z_j+O(|z|^2)$ as $z\to0$, $r+1\le j\le n$. There exists a real-analytic symplectomorphism $\Phi:(U\cap M,\omega)\hookrightarrow(\mathbb{R}^{2n},\omega_{can})$ such that $\Phi(m_0)=0$, $\mathrm{d}\Phi(m_0)=\mathrm{d}\Phi_0(m_0)$, and $J\circ F\circ\Phi^{-1}=(h_1,\dots,h_n)$, where $J_s(z):=z_s$ for $1\le s\le r$, $J=(J_1,\dots,J_n):W\to\mathbb{C}^n$. 
For a fixed $J$, any two such symplectomorphisms $\Phi,\Phi'$ near $m_0$ are related by $(\Phi^{-1}\circ\Phi')^2=\phi_{S\circ J\circ F}^1$, for some real-analytic function $S=S(z_1,\dots,z_n)$, where $\phi_f^t$ denotes the Hamiltonian flow generated by the function $f$.

{\rm(b)} The above $(S^1)^{n-r}$-action and its normalization are persistent and rigid (resp.) under real-analytic integrable perturbations in the following sense.
Suppose we are given a neighbourhood $U_1$ of $m_0$ in $M^{\mathbb C}$ and a neighbourhood $W_1$ of the origin in ${\mathbb C}^n$ having compact closures $\overline{U_1}\subset U$ and $\overline{W_1}\subset W$, and an integer $k\in{\mathbb Z}_+$.
Then there exists $\varepsilon>0$ such that, for any (``perturbed'') real-analytic integrable Hamiltonian system $(U\cap M,\tilde\omega,\tilde F)$ whose holomorphic extension to $U$ is $\varepsilon-$close to $(U,\omega^{\mathbb C},F^{\mathbb C})$ in $C^0$-norm, the following properties hold. On some neighbourhood $\tilde U\supseteq U_1$, there exists a unique $\tilde F^{\mathbb C}$-preserving Hamiltonian (w.r.t.\ the ``perturbed'' symplectic structure $\tilde\omega$) $(S^1)^{n-r}$-action generated by functions 
\begin{equation} \label {eq:X::}
\begin{array}{c}
\tilde J_{r+1}\circ\tilde F,\dots,\tilde J_{r+k_e}\circ\tilde F, \quad
i\tilde J_{r+k_e+1}\circ \tilde F,\dots,i\tilde J_{r+k_e+k_h}\circ \tilde F, \\
i\tilde J_{r+k_e+k_h+1}\circ \tilde F,\tilde J_{r+k_e+k_h+2}\circ \tilde F,\quad \dots,\quad
i\tilde J_{n-1}\circ \tilde F,\tilde J_n\circ \tilde F,
\end{array}
\end{equation}
where $\tilde J_j(z_1,\dots,z_n)$, $r+1\le j\le n$, are real-analytic functions on some neighbourhood $\tilde W\supseteq W_1$ that are $O(\varepsilon)-$close to $J_j(z_1,\dots,z_n)$ in $C^k$-norm.
There exists a real-analytic symplectomorphism $\tilde\Phi:(\tilde U\cap M,\tilde\omega)\hookrightarrow(\mathbb{R}^{2n},\omega_{can})$ whose holomorphic extension to $\tilde U$ is $O(\varepsilon)-$close to $\Phi^{\mathbb C}$ in $C^k$-norm such that $\tilde J\circ\tilde F\circ\tilde\Phi^{-1}=(h_1,\dots,h_n)$, where 
 $\tilde J_s(z)=z_s$ for $1\le s\le r$, $\tilde J=(\tilde J_1,\dots,\tilde J_n):\tilde W\to\mathbb{R}^n$.
If the system depends on a local parameter (i.e.\ we have a local family of systems), 
moreover its holomorphic extension to $M^{\mathbb C}$ depends smoothly (resp., analytically) on that parameter, then
$J$ and $\Phi$ can also be chosen to depend smoothly (resp., analytically) on that parameter.
\end{Lemma}

\begin{proof}
(a) We divide the proof into two steps.

{\em Step 1.} We can extend the functions $f_1,\dots,f_r$ to a system of local canonical coordinates $\Phi=(\lambda,\varphi,x,y)=(\lambda_1,\varphi_1,\dots,\lambda_r,\varphi_r,x_1,y_1,\dots,x_{n-r},y_{n-r}):U\hookrightarrow\mathbb{R}^{2n}$ 
on a small neighbourhood $U$ of $m_0$
such that $f_s=\lambda_s$ for $1\le s\le r$, $\Phi(m_0)=0$, $\mathrm{d}\Phi(m_0)=\mathrm{d}\Phi(m_0)$ and $\omega|_U=\Phi^*\omega_{can}$ (Darboux coordinates), see \eqref {eq:eli:}.
Consider two cases.

{\em Case 1:} $r=0$.
Consider the 1-parameter family of ``rescaling'' coordinate systems $\Phi_\varepsilon=(x',y')$ such that $x=\varepsilon x'$, $y=\varepsilon y'$, where $\varepsilon>0$ is a small parameter. Without loss of generality, we can and will assume that $F(m_0)=0$. By Hadamard's lemma, 
$$
f_j\circ\Phi_\varepsilon^{-1}(x',y')=f_j\circ\Phi^{-1}(\varepsilon x',\varepsilon y')=\varepsilon^2 f'_j(x',y',\varepsilon)
$$
for some real-analytic functions $f'_j(x',y',\varepsilon)$.
Clearly, $\omega=\varepsilon^2\Phi_\varepsilon^*(\mathrm{d} x'\wedge\mathrm{d} y')$ where $\mathrm{d} x'\wedge\mathrm{d} y':=\sum\limits_{j=1}^n\mathrm{d} x_j'\wedge\mathrm{d} y_j'$.
Choose a small $\varepsilon_0>0$ such that $\Phi(U)\supseteq B_{0,\varepsilon_0}:=\{w\in\mathbb{R}^{2n}\mid|w|<\varepsilon_0\}$. Denote $U_\varepsilon:=\Phi^{-1}(B_{0,\varepsilon})$ for $0<\varepsilon\le\varepsilon_0$.
Thus the rescaling diffeomorphism $\Phi_\varepsilon:U_\varepsilon\to B_{0,1}$ transforms the integrable Hamiltonian system
\begin{equation} \label {eq:scale:}
(U_\varepsilon,\varepsilon^{-2}\omega,\varepsilon^{-2}F)
\end{equation}
to the integrable system
\begin{equation} \label {eq:scale}
(B_{0,1},\mathrm{d} x'\wedge\mathrm{d} y',F'), 
\end{equation}
where $F\circ\Phi_\varepsilon^{-1}=\varepsilon^2F'(x',y',\varepsilon)$, 
$F'(x',y',\varepsilon):=(f_1'(x',y',\varepsilon),\dots,f_n'(x',y',\varepsilon))$.

Observe that the ``unperturbed'' system (i.e.\ \eqref {eq:scale} with $\varepsilon=0$)
\begin{equation} \label {eq:unperturb}
(B_{0,1},\mathrm{d} x'\wedge\mathrm{d} y',F'|_{\varepsilon=0}=(h_1,\dots,h_n))
\end{equation}
coincides with the linearization of the original system at $m_0$, which has the canonical form by assumption of the lemma.
Thus, on a small neighbourhood $U=U_\varepsilon$ of $m_0$, our system \eqref{eq:scale:} can be viewed as a system \eqref{eq:scale} obtained from the (canonical) ``unperturbed'' system \eqref {eq:unperturb} by $O(\varepsilon)$-small integrable perturbation.
Using this and \cite[Lemma 2.3]{kud:toric}, one can show that, for each $S^1$-subaction of the above $(S^1)^{n}$-action on $(T_{m_0}M)^\mathbb{C}$ (see Step 1), there exists a point $m_1\subset \mathcal{L}^\mathbb{C}$ satisfying the conditions (i)--(iii) of \cite[Theorem 2.2(a)]{kud:toric}.
By \cite[Theorem 2.2(a)]{kud:toric}, on a small open complexification $U^\mathbb{C}$ of $U=U_\varepsilon$, there exists a $F$-preserving Hamiltonian $(S^1)^{n}$-action generated by some functions \eqref {eq:X:} where $J_j=J_j(z_1,\dots,z_n)$ are real-analytic functions such that $J_j(z_1,\dots,z_n)=z_j+O(|z|^2)$, $1\le j\le n$.\footnote{Indeed: $J(F(\Phi^{-1}(x,y)))=J(F(\Phi_\varepsilon^{-1}(x',y')))=J(\varepsilon^2F'(x',y',\varepsilon))=\varepsilon^2J'(F'(x',y',\varepsilon),\varepsilon)$ for some real-analytic map $J'(z,\varepsilon)$ such that $J(\varepsilon^2z)=\varepsilon^2J'(z,\varepsilon)$.
Hence the rescaling diffeomorphism $\Phi_\varepsilon^\mathbb{C}$ conjugates the $(S^1)^n$-action $(U^\mathbb{C},\varepsilon^{-2}\omega^\mathbb{C},\varepsilon^{-2}J^\mathbb{C}\circ F^\mathbb{C})$ with the $(S^1)^n$-action $(B_{0,1}^\mathbb{C},(\mathrm{d} x'\wedge\mathrm{d} y')^\mathbb{C},{J'}^\mathbb{C}\circ {F'}^\mathbb{C})$. Hence the linearization of $X_{J_j\circ F}$ at $m_0$ is $\mathrm{d}\Phi_\varepsilon(m_0)$-conjugated with the linearization of $X_{J_j'\circ F'}$ at $0$ for any $\varepsilon>0$. 
But the quadratic part of $F'$ at $x'=y'=0$ does not depend on $\varepsilon$ and coincides with $F'|_{\varepsilon=0}=(h_1,\dots,h_n)$, see \eqref{eq:unperturb}.
This implies that $J_j'(z_1,\dots,z_n)=z_j+O(|z|^2)$ and, hence, $J_j(z_1,\dots,z_n)=z_j+O(|z|^2)$, $1\le j\le n$.}

By \cite[Theorem 3.10(a) or Lemma 6.2(a)]{kud:toric}, the latter $(S^1)^n$-action is linearizable at $m_0$ (\cite[Def.~3.1, 3.7]{kud:toric}). In other words, there exists a real-analytic symplectomorphism $\hat\Phi:(U,\omega)\hookrightarrow(\mathbb{R}^{2n},\omega_{can})$ 
sending the point $m_0$ to the origin, with $\mathrm{d}\hat\Phi(m_0)=\mathrm{d}\Phi(m_0)$, and transforming the momentum map $J\circ F$ to a collection of quadratic functions on $V=\Phi(\hat U)$, which does not depend on $\varepsilon$ and, hence, coincides with $(h_1,\dots,h_n)$ from \eqref {eq:unperturb}. 
Thus $J$ and $\hat\Phi$ have the required properties.

{\em Case 2:} $r>0$. One performs a local Hamiltonian reduction and reduces the problem to an $r$-parameter family of integrable systems with $n-r$ degrees of freedom, with a non-degenerate rank-$0$ point $m_0$. This can be done by the same arguments as in the case of a compact orbit $\mathcal{O}$ (see \cite[Sec.~4]{zung:mir04} or \cite[Sec.~7]{kud:toric}).

In detail: on a small neighbourhood $U_0$ of $m_0$, we can extend the functions $f_1,\dots,f_r$ to a system of local canonical coordinates 
$$
\Phi_0=(\lambda,\varphi,x,y)=(\lambda_1,\varphi_1,\dots,\lambda_r,\varphi_r,x_1,y_1,\dots,x_{n-r},y_{n-r}):U_0\hookrightarrow\mathbb{R}^{2n}
$$
such that $f_s=\lambda_s$ for $1\le s\le r$, $\Phi_0(m_0)=0$, and $\omega|_U=\Phi^*\omega_{can}$ (Darboux coordinates), see \eqref {eq:eli:}.
Take a local disk $P=\{\varphi=0\}$ of dimension $2n-r$ that intersects the local orbit $\mathcal{O}$ through $m_0$ transversally at $m_0$. 
Then the local disk $\{\lambda_1=\mathrm{const},\dots,\lambda_r=\mathrm{const}\}\cap P$ near $m_0$ has an induced symplectic structure and induced functions $f_{r+1},\dots,f_{n}$ that pairwise Poisson commute.

Applying the case of a rank-$0$ point and parameters $\lambda_1,\dots,\lambda_r$, which is a parametric extension of Case 1 (such an extension is valid due to the parametric extensions 
\cite[Theorems 2.2(b) and 3.10(b) or Lemma 6.2(b)]{kud:toric} of 
\cite[Theorems 2.2(a) and 3.10(a) or Lemma 6.2(a)]{kud:toric}), we can define 
an $F$-preserving Hamiltonian $(S^1)^{n-r}$-action on $P^\mathbb{C}$ 
and local functions $\hat x_1,\hat y_1,\dots,\hat x_{n-r},\hat y_{n-r}$ on $P$, such that they form a local symplectic coordinate system on each local disk $\{\lambda_1=\mathrm{const},\dots,\lambda_r=\mathrm{const}\}\cap P$, with respect to which the Hamiltonian $(S^1)^{n-r}$-action is linear and does not depend on the values of $\lambda_1,\dots,\lambda_r$. Moreover we have $\hat x(m_0)=\hat y(m_0)=0$, the local coordinates $(\hat x,\hat y)$ on the local disk $\{\lambda=0\}\cap P$ have the same linearization at $m_0$ as $(x,y)$.
We extend $\hat x_1,\hat y_1,\dots,\hat x_{n-r},\hat y_{n-r}$ to functions on $U$ by making them invariant under the local Hamiltonian flows of $X_{f_1},\dots,X_{f_r}$.

Since $\mathrm{d}\omega=0$, it follows \cite[Lemma 4.2]{zung:mir04} that the symplectic structure $\omega$ on $U$ has the form
$$
\omega = \sum_{s=1}^r\mathrm{d} \lambda_s\wedge\mathrm{d}(\varphi_s + g_s) + \sum_{j=1}^{n-r} \mathrm{d}\hat x_j\wedge\mathrm{d}\hat y_j ,
$$
for some real-analytic functions $g_s$ on a neighbourhood of $m_0$ in $U$, that are invariant under the local Hamiltonian flows of $X_{f_1},\dots,X_{f_r}$.

Define $\hat\lambda_s:=\lambda_s=f_s$, $\hat\varphi_s:=\varphi_s+g_s$, and $J_s(z_1,\dots,z_n)=z_s$ for $1\le s\le r$. Then with respect to the coordinate system $\hat\Phi=(\hat\lambda,\hat\varphi,\hat x,\hat y)$ on $U$, the symplectic form $\omega$ on $U$ has the standard form and the Hamiltonian $(S^1)^{n-r}$-action on $P^\mathbb{C}$ is linear and does not depend on $\lambda$. This implies that $\omega=\hat\Phi^*\omega_{can}$ and the functions
$$
J_{r+1}\circ F\circ\hat\Phi^{-1},\dots,J_{r+k_e}\circ F\circ\hat\Phi^{-1}, \
iJ_{r+k_e+1}\circ F\circ\hat\Phi^{-1},\dots,iJ_{r+k_e+k_h}\circ F\circ\hat\Phi^{-1},
$$
$$
J_{r+k_e+k_h+1}\circ F\circ\hat\Phi^{-1},iJ_{r+k_e+k_h+2}\circ F\circ\hat\Phi^{-1},\dots, J_{n-1}\circ F\circ\hat\Phi^{-1},iJ_{n}\circ F\circ\hat\Phi^{-1}
$$
generating this linear Hamiltonian $(S^1)^{n-r}$-action are quadratic functions in $x_j,y_j$ and do not depend on $\lambda$. 
Clearly, the functions
$$
J_{1}\circ F\circ\hat\Phi^{-1},\dots,J_{n}\circ F\circ\hat\Phi^{-1}
$$
have the canonical form \eqref{eq:eli}, and by construction
$$
J_s(z_1,\dots,z_n)=z_s \quad \mbox{for } 1\le s\le r, \qquad
J_s(z_1,\dots,z_n)=z_s+O(|z|^2) \quad \mbox{for } r+1\le s\le n.
$$

{\em Step 2.} It remains to prove the last assertion of (a). We will prove it for $r=0$ (the case $r>0$ can be reduced to the case $r=0$ by a local Hamiltonian reduction, as in Step 1).

Suppose $\Phi,\Phi'$ are two local symplectomorphisms at $m_0$ bringing $J\circ F$ to the canonical form. Then $\Psi:=(\Phi^{-1}\circ\Phi')^2$ is a $F$-preserving real-analytic symplectomorphism of a neighbourhood of $m_0$ to $M$ fixing $m_0$ and being homotopic to the identity in the space of $F$-preserving homeomorphisms. 
Take a regular point $m_1\in \mathcal{L}^\mathbb{C}$ close to $m_0$. 
Consider the rescaling diffeomorphism $\Phi_\varepsilon:U_\varepsilon\to B_{0,1}$ from Step 1. By Hadamard's lemma, the map $\Phi\circ\Psi\circ\Phi_\varepsilon^{-1}:B_{0,1}\to B_{0,1}$ has the form $\varepsilon\Psi_\varepsilon$, where $\Psi_\varepsilon:B_{0,1}\to B_{0,1}$ is a 1-parameter family of real-analytic maps in $x',y',\varepsilon$. Clearly, $\Psi_\varepsilon$ preserves $J'(F'(x',y',\varepsilon),\varepsilon)$ and $\mathrm{d} x'\wedge\mathrm{d} y'$. One checks that the unperturbed map $\Psi_0$ is linear and coincides with $\mathrm{d}(\Phi\circ\Psi\circ\Phi^{-1})(m_0)$.
Take a point $m'=(x',y')\in B_{0,1}^\mathbb{C}$ which is a regular point of the singular fiber of the unperturbed system \eqref{eq:unperturb}. Without loss of generality, $m'$ is fixed under the unperturbed linear map $\Psi_0$ (this can be achieved by replacing $\Psi$ with its composition with the time-$1$ map of the Hamiltonian flow generated by a linear combination of $J_j\circ F$, $1\le j\le n$). We can extend the functions $\lambda_j={J_j'}^\mathbb{C}\circ{F'}^\mathbb{C}$ to a local system of canonical holomorphic coordinates $\lambda_j,\mu_j$ near $m'$ (Darboux coordinates) depending analytically on $\varepsilon$ such that $\mu_j(m')=0$. Since $m'$ is fixed under the unperturbed map $\Psi_0$, it follows that, in these coordinates the perturbed map $\Psi_\varepsilon^\mathbb{C}$ on a neighbourhood $\tilde U(m')^\mathbb{C}$ of the point $m'$ has the form $(\lambda,\mu)\mapsto(\lambda,\mu+\frac{\partial S_\varepsilon^\mathbb{C}}{\partial\lambda})$, for some real-analytic function $S_\varepsilon=S_\varepsilon(\lambda)$ on a neighbourhood of the origin such that $S_0(0)=0$ and $\frac{\partial S_0(0)}{\partial\lambda_j}=0$. Thus, on $\tilde U(m')^\mathbb{C}$, the map $\Psi_\varepsilon^\mathbb{C}$ coincides with the time-1 map of the Hamiltonian flow generated by $S_\varepsilon^\mathbb{C}({J_j'}^\mathbb{C}({F'}^\mathbb{C}(x',y',\varepsilon),\varepsilon))$.
Choose a point $m_\varepsilon'=(x_\varepsilon',y_\varepsilon')\in\tilde U(m')^\mathbb{C}$ with $\lambda_j(m')=\mu_j(m')=0$.
Thus, $\Psi^\mathbb{C}$ coincides with the time-1 map of the Hamiltonian flow generated by $S\circ J\circ F$, where $S(z)=\varepsilon^2S_\varepsilon(z/\varepsilon^{2})$ on a small neighbourhood of the point $m_\varepsilon:=\Phi^{-1}(\varepsilon m_\varepsilon')\in \mathcal{L}^\mathbb{C}$.
Since the analytic symplectomorphisms $\Psi^\mathbb{C}$ and $(\phi_{S\circ J\circ F}^1)^\mathbb{C}$ are well-defined on some neighbourhood $U$ of $m_0$ in $M^\mathbb{C}$ and coincide with each other on a neighbourhood of the path $m_u\in \mathcal{L}^\mathbb{C}\cap U$, $0\le u\le\varepsilon$, by analytic continuation they must coincide on the whole $U$.

This yields Lemma \ref {lem:period} (a).

(b) On a neighbourhood $\tilde U$ of $m_0$ close to $U$, we can extend the ``perturbed'' functions $\tilde f_1,\dots,\tilde f_r$ to a ``perturbed'' system of local canonical coordinates $\tilde \Phi=(\tilde \lambda,\tilde \varphi,\tilde x,\tilde y):\tilde U\hookrightarrow\mathbb{R}^{2n}$ such that $\tilde f_s=\tilde \lambda_s$ for $1\le s\le r$ and $\tilde\omega|_{\tilde U}=\tilde \Phi^*\omega_{can}$ (perturbed Darboux coordinates), see \eqref {eq:eli:}.

We obtain a ``perturbed'' $r$-parameter family of integrable systems with $n-r$ degrees of freedom, with parameters $\tilde\lambda_1,\dots,\tilde\lambda_r$. Since the ``unperturbed'' system with zero values of the parameters ($\lambda_1=\dots=\lambda_r=0$) admits a non-degenerate rank-$0$ point $m_0$, we can derive the assertion (b) from Case 1 of (a) similarly to deriving Case 2 of (a), by applying to the ``perturbed'' system the ``perturbative'' extension 
\cite[Theorem 2.2(b) and 3.10(b) or Lemma 6.2(b)]{kud:toric} of 
\cite[Theorem 2.2(a) and 3.10(a) or Lemma 6.2(a)]{kud:toric}.

This yields Lemma \ref {lem:period} (b).
\end{proof}

\subsection {Proof of Theorem \ref {thm:2}}
(a) We want to bring our system to a canonical form \eqref{eq:eli}, \eqref{eq:eli:} on a small neighbourhood $U$ of the point $m_0$.
This can be done using \cite{vey}. Let us give another proof based on Lemma \ref {lem:period} (a) (which we proved using \cite{kud:toric}).

After replacing $f_1,\dots,f_n$ by their linear combinations, we can assume that 
$\mathrm{d} f_j(m_0)=0$ for each $j>r$. In particular, $\mathrm{d} f_1\wedge\dots\wedge\mathrm{d} f_r\ne0$ at $m_0$. Suppose also that $F(m_0)=0$ for $j>r$, which can be achieved by adding a constant to each $f_j$.

On a small neighbourhood $U$ of $m_0$, we can extend the functions $f_1,\dots,f_r$ to a system of local canonical coordinates $\Phi_0=(\lambda,\varphi,x,y)=(\lambda_1,\varphi_1,\dots,\lambda_r,\varphi_r,x_1,y_1,\dots,x_{n-r},y_{n-r}):U\hookrightarrow\mathbb{R}^{2n}$ such that $f_s=\lambda_s$ for $1\le s\le r$, $\Phi_0(m_0)=0$, and $\omega|_U=\Phi_0^*\omega_{can}$ (Darboux coordinates), see \eqref {eq:eli:}.

It follows from the Williamson theorem that (after replacing $f_{r+1},\dots,f_n$ by their linear combinations, and applying to $x,y$ a linear canonical transformation if necessary) the second differentials of $f_{r+1}\circ\Phi_0^{-1}|_{\lambda=0},\dots,f_n\circ\Phi_0^{-1}|_{\lambda=0}$ at the origin have a canonical form, i.e.\ coincide with the second differentials of $h_{r+1}|_{\lambda=0},\dots,h_n|_{\lambda=0}$ in \eqref{eq:eli}.
In particular, the linearizations at the point $m_0$ of the restrictions of the Hamiltonian vector fields generated by \eqref{eq:X:} to $\{\lambda=0\}$ have $2\pi$-periodic flows on $(T_{m_0}M)^\mathbb{C}$.

Due to Lemma \ref {lem:period} (a), there exist $J$ and $\Phi$ with required properties.

(b) Suppose $\mathcal{O}$ is a rank-$r$ orbit, $m_0\in\mathcal{O}$. Since the flows of all $X_{f_i}$ are complete on $\mathcal{O}$, it is diffeomorphic to a cylinder $\mathbb{R}^{r_o}\times(S^1)^{r_c}$, where $r_o$ and $r_c$ are {\em degree of openness} and {\em degree of closedness} of $\mathcal{O}$, respectively \cite[Def.~3.4]{zung96a}, $r=r_o+r_c$. 

By \cite{ito91} or \cite{zung02, zung05} (or \cite[Theorem 6.1]{zung96a} in the $C^\infty$ case with a proper $F$), there exists a locally-free $F$-preserving Hamiltonian $(S^1)^{r_c}$-action on a neighbourhood $U(\mathcal{O})$ of $\mathcal{O}$. 
This $(S^1)^{r_c}$-action is generated by functions of the form $J_{r_o+1}\circ F,\dots,J_r\circ F$ for some real-analytic functions $J_s(z_1,\dots,z_n)$, $r_o+1\le s\le r$.
Without loss of generality, $\partial(J_{r_o+1},\dots,J_r)/\partial(z_{r_o+1},\dots,z_r)\ne0$ and 
$\mathrm{d} f_1\wedge\dots\wedge\mathrm{d} f_r\ne0$ at some (and hence each) point of $\mathcal{O}$.
Without loss of generality, this $(S^1)^{r_c}$-action is effective.

Besides, we can extend to $U(\mathcal{O})^\mathbb{C}$ the Hamiltonian $(S^1)^{n-r}$-action on $U^\mathbb{C}$ generated by $J_j^\mathbb{C}\circ F^\mathbb{C}$, $r+1\le j\le n$, constructed in (a). 
The above $(S^1)^{r_c}$-action and $(S^1)^{n-r}$-action give rise to the Hamiltonian $(S^1)^{n-r_o}$-action on $U(\mathcal{O})^\mathbb{C}$ generated by $J_j^\mathbb{C}\circ F^\mathbb{C}$, $r_o+1\le j\le n$. Put $J_s(z_1,\dots,z_n):=z_s$, $1\le s\le r_o$. Consider two cases.

{\em Case 1:} $r_o=0$, thus the orbit $\mathcal{O}$ is compact. By \cite[Theorem 3.10(a)]{kud:toric}, the above $(S^1)^{n-r_o}$-action is symplectomorphic to a linear model, thus the system $(U(\mathcal{O}),\omega,J\circ F)$ is symplectomorphic to a linear model $(V/\Gamma,\omega_{can},(h_1,\dots,h_n))$ \cite[Def.~3.7]{kud:toric} having the form \eqref{eq:eli}, \eqref{eq:eli:}. In terminology of \cite[Def.~3.7]{kud:toric}, this means that the integer $(n-r)\times(n-r)$-matrix $\|p_{j\ell}\|$ (whose columns are ``extended'' elliptic and hyperbolic resonances of the singularity) is the unity matrix: $p_{j\ell}=\delta_{j\ell}$ (we can achieve this, since our matrix $\|p_{j\ell}\|$ is a non-degenerate square matrix, and we are allowed to replace the functions $J_j\circ F$ by their linear combinations forming a non-degenerate matrix). We can manage that the action of $\Gamma$ is trivial on each elliptic disk $D^2$ and on each focus-focus polydisk $D^2\times D^2$, because the twisting resonances are well-defined only up to adding any linear combinations of the ``extended'' elliptic resonances \cite[Remark 3.11(C)]{kud:toric}.
The action of $\Gamma$ on $(D^2)^{n-r}$ is effective, since otherwise the above $(S^1)^{r_c}$-action is non-effective.

{\em Case 2:} $r_o>0$. We deduce this case from a parametric version of Case 1 (similarly to the proof of Lemma \ref {lem:period} (a), Step 1, Case 2) by considering the corresponding reduced integrable Hamiltonian system with $n-r_o$ degrees of freedom (obtained by local symplectic reduction under the local Hamiltonian action generated by $f_1,\dots,f_{r_o}$). 
In this way, we see from Case 1 and \cite[Theorem 3.10(b)]{kud:toric} that the system $(U(\mathcal{O}),\omega,J\circ F)$ is symplectomorphic to a neighbourhood of the cylinder $\{0\}^r\times\mathbb{R}^{r_o}\times(S^1)^{r_c}\times\{(0,0)\}^{n-r}$ in the linear model $(V/\Gamma,\omega_{can},(h_1,\dots,h_n))$ having the form \eqref{eq:eli}, \eqref{eq:eli:}, as required.

This yields Theorem \ref {thm:2}. \qed

\smallskip
The authors are grateful to Alexey Bolsinov for helpful comments on a preliminary version of the paper and to Anton Izosimov for informing us about his results on structural stability of focus singularities. The work on semilocal singularities (Theorems \ref {thm:stab:help} and \ref {thm:stab}, Corollary \ref {cor:stab:help}, Sec.~\ref{sec:proof}, \ref {sec:kov} and App.~\ref {sec:proof:}) was supported by the Russian Science Foundation (grant No.~17-11-01303). 
The work on local singularities (Theorems \ref {thm:2} and \ref {thm:stab:}, Sec.~\ref{sec:rank0} and App.~\ref {sec:app}) was supported by the Russian Foundation for Basic Research (grant No.~19-01-00775-a).

\end{document}